\documentclass[runningheads]{llncs}
\usepackage[utf8]{inputenc}
\usepackage{amsmath,amsfonts,amssymb}
\usepackage{hyperref}
\usepackage{cleveref}
\usepackage{tikz}
\usepackage{blkarray}
\usetikzlibrary{calc}
\usepackage{caption}
\usepackage{subcaption}

\usepackage[margin=1.4in]{geometry}

\newcommand{\R}{\mathbb{R}}

\renewcommand{\subset}{\subseteq}

\newcommand{\conv}{\mathop\mathrm{conv}\nolimits}

\crefname{theorem}{Theorem}{Theorems}
\crefname{lemma}{Lemma}{Lemmas}
\crefname{corollary}{Corollary}{Corollaries}
\crefname{proposition}{Proposition}{Propositions}
\crefname{conjecture}{Conjecture}{Conjectures}
\crefname{section}{Section}{Sections}
\crefname{figure}{Figure}{Figures}
\crefname{definition}{Definition}{Definitions}

\spnewtheorem*{theorem*}{Theorem}{\normalshape\bfseries}{\itshape}
\newtheorem{oq}[theorem]{Open Question}

\usepackage{enumitem}
\newlist{myenumerate}{enumerate}{1}
\setlist[myenumerate]{
  label={\itshape(\roman*)},  % label format (i),(ii),...
  ref={\itshape(\roman*)},    % same reference format
}
\crefname{myenumeratei}{}{}

\crefname{enumi}{}{}

\usepackage{todonotes}

\title{On the Circuit Diameter Conjecture for\\ Counterexamples to the Hirsch Conjecture  \thanks{%
    This work was completed in part at the 2022 Graduate Research Workshop in Combinatorics, which was supported in part by NSF grant 1953985, and a generous award from the Combinatorics Foundation.\\ The first author was supported by the NSF GRFP and NSF DMS-1818969.
    The second author was supported by Air Force Office of Scientific Research grant FA9550-21-1-0233, NSF grant 2006183, Algorithmic Foundations, Division of Computing and Communication Foundations, and Simons Foundation grant 524210. 
    The third author was partially supported by the Alexander von Humboldt Foundation with funds from the German Federal Ministry of Education and Research (BMBF).
}}
\titlerunning{ }

\author{%
    Alexander~E.~Black\inst{1}\orcidID{0000-0002-7445-5820} \and
    Steffen~Borgwardt\inst{2}\orcidID{0000-0002-8069-5046} \and
    Matthias~Brugger\inst{3}\orcidID{0000-0003-1571-5239}
}
\authorrunning{ }

\institute{%
    University of California, Davis\\ \email{aeblack@ucdavis.edu}\\ \and
    University of Colorado Denver\\ \email{steffen.borgwardt@ucdenver.edu}\\ \and
    Technical University of Munich\\ \email{matthias.brugger@tum.de}
}

\begin{document}
\maketitle
\begin{abstract}
Circuit diameters of polyhedra are a fundamental tool for studying the complexity of circuit augmentation schemes for linear programming and for finding lower bounds on combinatorial diameters. The main open problem in this area is the circuit diameter conjecture, the analogue of the Hirsch conjecture in the circuit setting. A natural question is whether the well-known counterexamples to the Hirsch conjecture carry over. Previously, Stephen and Yusun showed that the Klee-Walkup counterexample to the unbounded Hirsch conjecture does not transfer to the circuit setting. Our main contribution is to show that the original counterexamples for the other variants, for bounded polytopes and using monotone walks, also do not transfer. 

Our results rely on new observations on structural properties of these counterexamples. To resolve the bounded case, we exploit the geometry of certain $2$-faces of the polytopes underlying all known bounded Hirsch counterexamples in Santos' work. For Todd's monotone Hirsch counterexample, we provide two alternative approaches. The first one uses sign-compatible circuit walks, and the second one uses the observation that Todd's polytope is anti-blocking. Along the way, we enumerate all linear programs over the polytope and find four new orientations that contradict the monotone Hirsch conjecture, while the remaining $7107$ satisfy the bound.

\end{abstract}

%---------------------------------------------------------------------
\section{Introduction}

Relating the diameter of a polyhedron to its dimension and its number of facets is a classical topic in optimization. The {\em combinatorial diameter} of a polyhedron is the maximum length of a shortest path between any two vertices in the graph (or $1$-skeleton) of the polyhedron. The famous Hirsch conjecture \cite{d-63} claimed a bound of $f-d$ on the combinatorial diameter of any $d$-dimensional polyhedron with $f$ facets. It was first disproved for unbounded polyhedra \cite{kw-67} and in a stronger setting requiring monotone walks \cite{ToddExample}, and only disproved much later for polytopes \cite{msw-15,s-11}, i.e., bounded polyhedra. 
Today, the arguably most important open question in the field is the {\em polynomial Hirsch conjecture}, which asks whether there is a polynomial bound on the diameter in terms of $f$ and $d$. Note that the existence of a strongly polynomial pivot rule for the Simplex method would require this conjecture to be true. The same holds for a polynomial version of the {\em monotone Hirsch conjecture}, which concerns edge walks that are strictly decreasing and lead to a minimal vertex for some linear objective function.

To approach these long-standing questions, a number of abstractions and generalizations of edge walks on the $1$-skeleton have been introduced (see, e.g., \cite{cs-22,ContinuousHirsch,ehrr-10,KimAbstraction,ks-10,s-13} and references therein). We are interested in the {\em circuit diameter} of polyhedra introduced in \cite{bfh-14}. {\em Circuits} are a classical topic in oriented matroid theory \cite{oxley-06}, and correspond to the elementary vectors as introduced in \cite{r-69}. Informally, they can be interpreted as support-minimal linear dependencies between the columns of a constraint matrix. In this work, we study whether the original counterexamples to the Hirsch conjecture and the monotone Hirsch conjecture can be transferred to the circuit setting. We begin with some necessary background in Section \ref{sec:conjecture} and explain our contributions in Section \ref{sec:contributions}. 

\subsection{Circuit Diameters and the Circuit Diameter Conjecture}\label{sec:conjecture}

We follow \cite{bv-17,bv-22,dknv-22,dhl-15} for formal definitions and the important concepts. 

\begin{definition}[Circuits]
For a rational polyhedron $P = \{ \mathbf{x} \in \R^n \colon A \mathbf{x} = \mathbf{b}, B \mathbf{x} \leq \mathbf{d} \}$, the set of circuits of $P$, denoted $\mathcal{C}(A,B)$, consists of all vectors $\mathbf{g} \in \ker(A) \setminus \{ \mathbf{0} \}$, normalized to have coprime integer components, for which $B \mathbf{g}$ is support-minimal in the set $\{ B \mathbf{x} \colon \mathbf{x} \in \ker(A) \setminus \{ \mathbf{0} \}\}$.
\end{definition}

The set of circuits is precisely the set of potential edge directions of the polyhedron as the right hand sides $\mathbf{b}$ and $\mathbf{d}$ vary \cite{g-75}. In particular, it contains the set of all actual edge directions. Thus, a {\em circuit walk}, referring to a sequence of maximal steps along circuits, is a direct generalization of an edge walk.

 \begin{definition}[Circuit Walk]\label{def:circuitwalk}
 Let $P = \{ \mathbf{x} \in \R^n \colon A \mathbf{x} = \mathbf{b}, B \mathbf{x} \leq \mathbf{d} \}$ be a polyhedron. For two vertices $\mathbf{v}_1$ and $\mathbf{v}_2$ of $P$, we call a sequence $\mathbf{v}_1 = \mathbf{y}_0,\dots,\mathbf{y}_k = \mathbf{v}_2$ a circuit walk of length $k$ from $\mathbf{v}_1$ to $\mathbf{v}_2$ in $P$ if, for all $i = 1,\dots,k$, 
 \begin{myenumerate}
     \item $\mathbf{y}_i \in P$,
     \item $\mathbf{y}_{i} = \mathbf{y}_{i-1} +\alpha_i \mathbf{g}_i$ for some $\mathbf{g}_i \in \mathcal{C}(A,B)$ and $\alpha_i > 0$, and
     \item $\mathbf{y}_{i-1} + \alpha \mathbf{g}_i \notin P$ for all $\alpha > \alpha_i$.
        \label{def:circuitwalk-3}
 \end{myenumerate}
 \end{definition}
 
 We define the {\em circuit diameter} of a polyhedron $P$ as the maximum length of a shortest circuit walk between any pair of vertices of $P$. Note that, unlike edge walks, circuit walks are not necessarily reversible: the number of steps required to walk from $\mathbf{v}_1$ to $\mathbf{v}_2$ may not be the same as from $\mathbf{v}_2$ to $\mathbf{v}_1$. We use $\Delta(f,d)$ to denote the maximum circuit diameter of any $d$-dimensional polyhedron with $f$ facets.  The circuit analogue of the Hirsch conjecture, the {\em circuit diameter conjecture} \cite{bfh-14}, asks whether $\Delta(f,d)\leq f-d$ and is open. 
 
 As with the Hirsch conjecture, there is a monotone variant of the circuit diameter conjecture. To start, 
we call a circuit walk $\mathbf{y}_0,\dots,\mathbf{y}_k$ {\em monotone} with respect to a linear objective function $\mathbf{c}$ if the sequence $(\mathbf{c}^\top \mathbf{y}_i)_{i=0,\dots,k}$ is strictly decreasing. 
The {\em monotone circuit diameter} of a polyhedron $P$ is defined as the maximum length of a shortest monotone circuit walk from any vertex of $P$ to a vertex minimizer of $\mathbf{c}$ across all possible choices of objective function $\mathbf{c}$. It is open whether the monotone circuit diameter is always bounded by $f-d$.

The studies of circuit diameters, monotone circuit diameters, and the associated conjectures arise in several ways. Clearly, the circuit diameter is a lower bound on the combinatorial diameter, and thus is studied as a proxy. In the same way that the combinatorial diameter relates to the possible efficiency of a primal Simplex method, (monotone) circuit diameters are intimately related to the efficiency of {\em circuit augmentation schemes} for linear programs \cite{bv-19b,bv-22,dhl-15,env-21,gdl-14,gdl-15}. Further, a resolution of the circuit diameter conjecture would reveal some information as to {\em why} the Hirsch bound of $f-d$ is violated in the combinatorial setting \cite{bdf-16,bsy-18}. More specifically, an affirmative answer to the circuit diameter conjecture implies that it is the restriction from circuit to edge steps that causes the violation. On the other hand, if not even the circuit diameter satisfies the Hirsch bound, the reason for this would be the maximality of steps in \cref{def:circuitwalk}\cref{def:circuitwalk-3}: if the step lengths are not required to be maximal, the so-called {\em conformal sum property} \cite{bk-84,r-69,z-95} guarantees the existence of a walk of at most $f-d$ circuit steps between any pair of vertices  (see also \cref{sec:shortcirc}). 

Despite results on circuit diameters for some polyhedra in combinatorial optimization (see, e.g., \cite{bdfm-18,bfh-16,kps-17}), and general upper bounds involving the so-called circuit imbalance \cite{dknv-22,env-21} or the input bit-size \cite{dks-22}, not much is known about the potential validity of the circuit diameter conjecture. It was shown in \cite{bsy-18} that $\Delta (8, 4) = 4$. In particular, this holds for the Klee-Walkup polyhedron in the original counterexample to the unbounded Hirsch conjecture \cite{kw-67}: an unbounded $4$-dimensional polyhedron with $8$ facets and combinatorial diameter $5 >8-4$, but circuit diameter at most $4$ \cite{sy-15}. The question that motivates our work is whether any of the other well-known counterexamples to variants of the Hirsch conjecture may be counterexamples to the respective circuit analogues.

\subsection{Contributions}\label{sec:contributions}

We study four polytopes that appear in the construction of (monotone) Hirsch counterexamples: the $5$-dimensional polytopes $S^{48}_5$ from \cite{s-11} and $S^{25}_5$, $S^{28}_5$ from \cite{msw-15} (named for their number of facets $f=48$ or $f=25,28$, respectively) that are the basis for a construction of counterexamples to the bounded Hirsch conjecture, as well as the $4$-dimensional polytope $M_4$ from \cite{ToddExample} underlying the original counterexample for the monotone Hirsch conjecture. Each of these polytopes is a so-called \emph{spindle}. A spindle is the intersection of two pointed cones emanating from two \emph{apices} $\mathbf{u}$ and $\mathbf{v}$ such that each apex is in the interior of the opposite cone.

Our first main contribution is a proof that Todd's counterexample $M_{4}$ to the monotone Hirsch conjecture is not a counterexample in the circuit setting. In Section \ref{sec:Todd}, we prove the following.

\begin{theorem*}
The monotone circuit diameter of $M_{4}$ is $4$.
\end{theorem*}

Todd showed for the orientation of the graph of $M_{4}$ induced by minimizing $(1,1,1,1)^{\top} \mathbf{x}$ that the worst-case combinatorial distance to the unique sink is $5$. In \cref{sec:Todd}, we study all $7112$ orientations of the graph of $M_{4}$ induced by linear objective functions and find that five orientations can be used to contradict the monotone Hirsch bound for the combinatorial diameter. We observe that these orientations all have a unique sink at $\mathbf{0}$ and leverage this observation for our arguments. Specifically, for any linear objective function minimized at $\mathbf{0}$, we exhibit the existence of a monotone circuit walk of length $4$. We do so in two more general settings. In Section \ref{sec:shortcirc}, we use sign-compatible walks to find short monotone circuit walks on a large class of spindles. In Section \ref{sec:antiblock}, we show that the Todd polytope is anti-blocking and argue the existence of short monotone circuit walks on anti-blocking polytopes. 

Our second main contribution is on Santos' original counterexamples to the bounded Hirsch conjecture. A key notion for Santos' arguments is the {\em length} of a spindle, the combinatorial distance between the apices. We will bound the {\em circuit length}, referring to the maximum length of a shortest circuit walk from one apex of the spindle to the other one. 

At the core of Santos' disproof of the bounded Hirsch conjecture in \cite{s-11} is the following observation:
from a $d$-dimensional degenerate spindle with $f$ facets and length greater than $d$, one can obtain a $(d{+}1)$-dimensional spindle with $f+1$ facets which has length greater than $d+1$. By doing this $f-2d$ times, one obtains an $(f{-}d)$-dimensional spindle with $2f-2d$ facets whose length exceeds $f-d$. Santos gave a highly degenerate $5$-dimensional spindle $S^{48}_5$ with $48$ facets and length $6$. He then concluded via his iterative construction that there is a $43$-dimensional (bounded) Hirsch counterexample, namely a spindle with $86$ facets and length at least $44$. In a follow-up to Santos' work, Matschke, Santos, and Weibel \cite{msw-15} found two spindles, $S^{28}_5$ and $S^{25}_5$, also of dimension $5$ and length $6$ but with fewer facets. These lead to counterexamples in lower dimensions $23$ and $20$, respectively. In \cref{sec:Santos}, we consider Santos' original spindle $S^{48}_5$ from \cite{s-11} as well as the two smaller ones $S^{28}_5$ and $S^{25}_5$ from \cite{msw-15} and prove the following: 

\begin{theorem*}
The circuit length of all three spindles $S^{48}_5$, $S^{28}_5$, and $S^{25}_5$ is at most $5$. The same bounds also hold for all realizations resulting from mild perturbations.
\end{theorem*}

This implies that Santos' construction applied to any of the three spindles, or mild perturbations thereof, does not lead to a counterexample in the circuit setting. To prove this result in \cref{sec:Santos-5}, we use the geometry of certain $2$-faces. Interestingly, our proof exhibits that for all three spindles the apices can be connected by circuit walks of length at most $5$ with no more than two non-edge steps. With the same arguments, in \cref{sec:Santos-highdim}, we are also able to verify computationally that the circuit length of the two explicit (high-dimensional) Hirsch counterexamples from \cite{msw-15}, obtained from carrying out the steps of Santos' construction starting with $S^{28}_5$ and $S^{25}_5$, is indeed at most their dimension. Thus, not only does the construction principle not transfer, but we are also able to see directly for the two devised high-dimensional polytopes that they do not disprove the circuit diameter conjecture through their circuit length. 

One of the main challenges in studying circuit diameters is the fact that the circuit diameter of a polyhedron, unlike the combinatorial diameter, depends on its geometry. In particular, there can be two realizations of the same polyhedron with different circuit diameters (see, for example, \cite{sy-15b}). In \cref{sec:Santos-5}, we provide sufficient conditions for the circuit length to be at most the dimension. We exhibit that these conditions are satisfied for the specific realizations of the spindles in the literature and mild perturbations. It will remain open whether {\em all} realizations of these polytopes satisfy these conditions, or more generally whether all realizations have a circuit length bounded by the dimension. 

Our analysis in \cref{sec:Santos} is restricted to the particular spindles used by Matschke, Santos, and Weibel to build counterexamples to the Hirsch conjecture. The circuit diameter conjecture for spindles remains interesting. 

\begin{oq}
\label{oq:circlength}
Does there exist a $d$-dimensional spindle with circuit length at least $d+1$?
\end{oq}
 
 Conjecture 3.8 of \cite{bsy-18} asks the same question for simple spindles. In Section \ref{sec:shortcirc}, we provide some conditions under which we can always find a circuit walk of length $d$ between the apices of a spindle, so there are some positive results in this direction. It is open whether Conjecture $3.8$ from \cite{bsy-18} or a negative answer to Open Question \ref{oq:circlength} implies the circuit diameter conjecture, but we believe that both implications hold. 

%---------------------------------------------------------------------
%\clearpage
\section{Monotone Hirsch Counterexamples}\label{sec:monotone}

In this section, we show that Todd's counterexample to the monotone Hirsch conjecture \cite{ToddExample} does not transfer to the circuit setting.

In Section \ref{sec:shortcirc}, we identify a close connection of sign-compatible and monotone walks on spindles for objective functions uniquely minimized at $\mathbf{0}$. We then exploit this connection in Section \ref{sec:Todd} to prove the monotone circuit diameter of $M_{4}$ is precisely $4$. In Section \ref{sec:antiblock}, we highlight an alternative avenue to the disproof by observing that $M_4$ is an anti-blocking polytope. 

\subsection{Short Circuit Walks on Spindles}\label{sec:shortcirc} 

We start by finding conditions that guarantee the existence of \emph{sign-compatible} circuit walks on spindles. Two vectors $\mathbf{x},\mathbf{y} \in \R^d$ are \emph{sign-compatible} if $\mathbf{x}_i \mathbf{y}_i \ge 0$ for all $i=1,\dots,d$. Let $P = \{ \mathbf{x} \in \R^d \colon B \mathbf{x} \le \mathbf{d} \}$ be a polyhedron for a matrix $B \in \R^{m \times d}$ with rows $\mathbf{b}_i \in \R^d$ for $i=1,\dots,m$.
We call a circuit walk on $P$ with steps $\mathbf{g}_j$ (and maximal step lengths $\alpha_j$) for $j=1,\dots,k$ a \emph{sign-compatible circuit walk} if all $B \mathbf{g}_j$ for $j=1,\dots,k$ are pairwise sign-compatible. 
Such walks are special cases of conformal sums of circuits, which correspond to sign-compatible circuit walks without the maximal step requirement in \cref{def:circuitwalk}\cref{def:circuitwalk-3}. In this weaker setting, the well-known \emph{conformal sum property} \cite{bk-84,r-69,z-95} guarantees that for any given pair of vertices $\mathbf{u},\mathbf{v}$ of a $d$-dimensional polyhedron, their difference $\mathbf{v}-\mathbf{u}$ can be written as a conformal sum of at most $d$ circuits. This contrasts with the situation for sign-compatible circuit walks (with maximal steps), which may not exist at all \cite{bdf-16}. When they do, however, their length is at most $d$: note that once a sign-compatible walk enters a new facet, it may not leave it again. Since each step of the walk is maximal and must therefore enter a new facet, the number of steps is at most $d$, and thus satisfies the bound $f-d\geq d$ for a spindle. 

Moreover, sign-compatible circuit walks are monotone for any linear objective function that is uniquely minimized at the ending vertex of the walk. To see this, let $P$ be a polyhedron given by $P = \{ \mathbf{x} \in \R^d \colon B \mathbf{x} \le \mathbf{d} \}$. For the sake of simplicity, we only consider full-dimensional polyhedra here. We define the hyperplane arrangement $\mathcal{H}(B) = \bigcup_{i=1}^{m} \{\mathbf{x} \in \R^d \colon \mathbf{b}_i^\top \mathbf{x} = 0\}$. Following \cite{bv-17}, we call it the {\em elementary arrangement} of $P$. 
Since the set of circuits of $P$ consists precisely of the (normalized) directions of the extreme rays of $\mathcal{H}(B)$ (see \cite{bv-17}), we obtain the following equivalent characterization of sign-compatibility:

\begin{lemma} \label{prop:sign-comp-char}
Let $P = \{ \mathbf{x} \in \R^d \colon B \mathbf{x} \le \mathbf{d} \}$ be a polyhedron in $\R^d$ and let $\mathbf{u}$ and $\mathbf{v}$ be two vertices of $P$. Denote by $C$ the minimal face of the elementary arrangement $\mathcal{H}(B)$ containing $\mathbf{v}-\mathbf{u}$.
For a circuit walk $\mathbf{u} = \mathbf{y}_0,\dots,\mathbf{y}_k = \mathbf{v}$ from $\mathbf{u}$ to $\mathbf{v}$ in $P$ with steps $\mathbf{g}_j \in \mathcal{C}(B)$ for $j=1,\dots,k$, the following statements are equivalent:
\begin{myenumerate}
    \item The walk $\mathbf{y}_0,\dots,\mathbf{y}_k$ is sign-compatible. \label{prop:sign-comp-char-1}
    \item $\mathbf{g}_j \in C$ for all $j=1,\dots,k$. \label{prop:sign-comp-char-2}
    \item The walk $\mathbf{y}_0,\dots,\mathbf{y}_k$ is monotone for any linear objective function $\mathbf{c} \in \R^d$ uniquely minimized over $-C$ at the origin $\mathbf{0}$.  \label{prop:sign-comp-char-3}
\end{myenumerate}
\end{lemma}
\begin{proof}
To see that \cref{prop:sign-comp-char-1} and \cref{prop:sign-comp-char-2} are equivalent, note that the set of all vectors in $\R^d$ whose products with $B$ are pairwise sign-compatible and sign-compatible with $B(\mathbf{v}-\mathbf{u})$ is a polyhedral cone (see \cite{bdf-16}) and coincides with the minimal face of $\mathcal{H}(B)$ containing $\mathbf{v}-\mathbf{u}$.

Let us now prove the equivalence of \cref{prop:sign-comp-char-2} and \cref{prop:sign-comp-char-3}. 
From standard polyhedral theory, $\mathbf{c}$ is uniquely minimized over $-C$ at $\mathbf{0}$ if and only if $-\mathbf{c}$ is in the relative interior of the polar cone of $-C$, which means that $\mathbf{c}^\top\mathbf{x} < 0$ for all $\mathbf{x} \in C \setminus \{\mathbf{0}\}$.  Hence, if $\mathbf{g}_{j} \in C$ for all $j = 1, \dots, k$ then $\mathbf{c}^{\top} \mathbf{g}_{j} < 0$ for all such $j$, meaning that the path is monotone for all choices of $\mathbf{c}$. Thus, \cref{prop:sign-comp-char-2} implies \cref{prop:sign-comp-char-3}. For the other direction, recall that polar duality is an involution, so $\mathbf{x} \in C$ if and only if, for all $\mathbf{c}$ uniquely minimized at $\mathbf{0}$ on $-C$, $\mathbf{c}^\top \mathbf{x} < 0$. Thus, if the walk is monotone for all linear objective functions uniquely minimized over $-C$ at $\mathbf{0}$, each step $\mathbf{g}_{j}$ must be contained in $C$, and so \cref{prop:sign-comp-char-3} implies \cref{prop:sign-comp-char-2}.  
\qed
\end{proof}

A key ingredient for our proofs is that we are able to guarantee the existence of sign-compatible circuit walks on spindles satisfying some restrictions. Recall that a spindle is the intersection of two pointed cones $C_{1}$ and $C_{2}$ such that the unique vertices $\mathbf{u}$ and $\mathbf{v}$ of $C_{1}$ and $C_{2}$ are contained in the interior of $C_{2}$ and $C_{1}$, respectively. We make the following observation: one can always find sign-compatible circuit walks between the two apices of a spindle formed by repeating the same cone twice.

\begin{lemma} \label{lem:spindle-same-cone} 
Let $P \subset \R^d$ be a spindle with apices $\mathbf{u}$ and $\mathbf{v}$, given by $(C + \mathbf{u}) \cap (-C + \mathbf{v})$ for a pointed cone $C$ with a unique vertex at the origin $\mathbf{0}$. Then there is a sign-compatible circuit walk of length at most $d$ from $\mathbf{u}$ to $\mathbf{v}$ in $P$.
\end{lemma}

\begin{proof}
We show a slightly stronger statement: we argue that for any point $\mathbf{x} \in P$ with $\mathbf{x} \neq \mathbf{v}$, we can follow a single circuit step of maximal length that makes a new facet-defining inequality of $P$ containing $\mathbf{v}$ tight, while ensuring that all facet-defining inequalities tight at both $\mathbf{x}$ and $\mathbf{v}$ remain tight. Furthermore, the direction of the circuit step is always one of the extreme rays of $C$. Since $C$ is a region of the elementary arrangement of $P$ and $\mathbf{v}-\mathbf{u}$ is in the interior of $C$, any circuit walk from $\mathbf{u}$ to $\mathbf{v}$ that only walks along directions of extreme rays of $C$ will be sign-compatible by \cref{prop:sign-comp-char}. Hence, this strategy suffices to prove the statement.

To do so, we construct a smaller spindle inside $P$ containing $\mathbf{x}$ as an apex. Consider $C + \mathbf{x}$. Note that $\mathbf{x} \in P \subseteq -C + \mathbf{v}$, so $\mathbf{x} = \mathbf{v} - \mathbf{y}$ for some $\mathbf{y} \in C$, which means that $\mathbf{v} \in \mathbf{x} + C$. Hence, $\mathbf{x}, \mathbf{v} \in (C+ \mathbf{x}) \cap (-C + \mathbf{v})$. Let $C'$ be the minimal face of $C$ such that $-C' + \mathbf{v}$ contains $\mathbf{x}$. Then $\mathbf{x}$ is on the relative interior of $-C' + \mathbf{v}$. It follows that $\mathbf{v}$ is also on the relative interior of $C' + \mathbf{x}$. 

Now consider $P' = (C'+\mathbf{x}) \cap (-C' +\mathbf{v})$ and note that $P'$ is a spindle where the same face $C'$ of $C$ is repeated twice. Let $\mathbf{y}$ be any vertex adjacent to $\mathbf{x}$ on $P'$. By construction, all inequalities of $P$ tight at both $\mathbf{v}$ and $\mathbf{x}$ must also be tight at $\mathbf{y}$. Further, the facet-defining inequalities for $P'$ are tight at exactly one of $\mathbf{x}$ and $\mathbf{v}$. Thus, one of the facet-defining inequalities tight at $\mathbf{v}$ and not $\mathbf{x}$ must be tight at $\mathbf{y}$. The inequalities tight at $\mathbf{v}$ in $P'$ are a subset of the inequalities tight at $\mathbf{v}$ in $P$. Hence, the step from $\mathbf{x}$ to $\mathbf{y}$ must be a circuit step of maximal step length in $P$. Furthermore, the direction of the step is an extreme ray of $C'$ and therefore an extreme ray of $C$. \qed 
\end{proof}

Lemma \ref{lem:spindle-same-cone} implies that any spindle with the same feasible cone at each apex will never yield a negative answer to Open Question \ref{oq:circlength}. As a generalization of this observation, we may find a short circuit walk in spindles where one of the feasible cones is contained in the other. 

\begin{theorem} \label{thm:spindle-containment}
Let $P \subset \R^d$ be a spindle with apices $\mathbf{u}$ and $\mathbf{v}$, given by $(C + \mathbf{u}) \cap (-D + \mathbf{v})$ for two pointed cones $C$ and $D$  with a unique vertex at the origin $\mathbf{0}$. Suppose that $D \subseteq C$. Then there is a sign-compatible circuit walk of length at most $d$ from $\mathbf{u}$ to $\mathbf{v}$ in $P$.
For any linear objective function $\mathbf{c} \in \R^d$ uniquely minimized over $P$ at $\mathbf{v}$, this walk is monotone.
\end{theorem}
\begin{proof}
Define $P' = (D + \mathbf{u}) \cap (-D + \mathbf{v})$. Since $P$ is a spindle, $\mathbf{v}-\mathbf{u}$ is in the interior of $C \cap D = D$ by hypothesis. It follows then that $P'$ is a spindle with apices $\mathbf{u}$ and $\mathbf{v}$. Since $D \subseteq C$, we further have that $P' \subseteq P$. By \cref{lem:spindle-same-cone}, there is a sign-compatible circuit walk of length at most $d$ steps from $\mathbf{u}$ to $\mathbf{v}$ in $P'$. Each step walks along the direction of some extreme ray of $D$ by \cref{prop:sign-comp-char}. Therefore, the only facet-defining inequalities of $P'$ that can become tight at each step are facet-defining inequalities for the cone $-D + \mathbf{v}$, which are facet-defining for $P$ as well. So each step is maximal in $P$ and the walk is therefore a circuit walk in $P$.

Sign-compatibility for $P$ follows from \cref{prop:sign-comp-char} by noting that, since $D \subseteq C$, $D$ is a region of the elementary arrangement which contains $\mathbf{v}-\mathbf{u}$ in its interior. For monotonicity, note that any linear objective function $\mathbf{c}$ for which $\mathbf{v}$ is the unique minimizer over $P$ is uniquely minimized at $\mathbf{0}$ over $-D$, since $-D$ is the feasible cone at $\mathbf{v}$. By \cref{prop:sign-comp-char}, the walk must therefore be monotone for any such $\mathbf{c}$. \qed
\end{proof}

Note that it is a fundamental assumption for our proof that we walk from the vertex $\mathbf{u}$ of the larger cone to the vertex $\mathbf{v}$ of the smaller cone. One may naturally consider whether there is a sign-compatible circuit walk going in the other direction from $\mathbf{v}$ to $\mathbf{u}$. It turns out that if such a circuit walk always existed, then we could bound the circuit length for all spindles. 

\begin{remark}
Suppose that for any spindle $(C+\mathbf{u}) \cap (-D + \mathbf{v})\subseteq \R^d$ such that $D \subseteq C$, there is a sign-compatible circuit walk from $\mathbf{v}$ to $\mathbf{u}$. Then any spindle has circuit length at most $d$.
\end{remark}

\begin{proof}  Consider a spindle $(C_{1} + \mathbf{u}) \cap (-C_{2} + \mathbf{v})$. Observe that, since $\mathbf{v}$ is on the interior of $C_{1} + \mathbf{u}$, $\mathbf{u}$ is on the interior of $-C_{1} + \mathbf{v}$. It follows that $\mathbf{u}$ is on the interior of $(-C_{1} + \mathbf{v}) \cap (-C_{2} + \mathbf{v}) = -(C_{1} \cap C_{2}) + \mathbf{v}$. Let $C = C_{1}$ and $D = C_{1} \cap C_{2}$. Then $D \subseteq C$, so by hypothesis, there exists a sign-compatible circuit walk of length at most $d$ from $\mathbf{v}$ to $\mathbf{u}$ in $(C + \mathbf{u}) \cap (-D + \mathbf{v})$. Note that the set of circuits of $(C + \mathbf{u}) \cap (-D+\mathbf{v})$ is a subset of the set of circuits of $(C_{1}+\mathbf{u}) \cap (-C_{2}+\mathbf{v})$, by construction. That sign-compatible circuit walk remains a circuit walk in $(C_{1} + \mathbf{u}) \cap (-C_{2} + \mathbf{v})$, since at each step the walk must enter a new facet of $C + \mathbf{u} = C_{1} + \mathbf{u}$ and can never return to a facet of $-D + \mathbf{v}=-(C_{1} \cap C_{2}) + \mathbf{v}\subseteq (-C_{2} + \mathbf{v})$. \qed
 \end{proof}

\subsection{Todd's Monotone Hirsch Counterexample} \label{sec:Todd}

It remains to apply the results of Section \ref{sec:shortcirc} to the polytope $M_4$ used in the disproof of the monotone Hirsch conjecture. 

The Todd polytope $M_{4}$ is given by $M_{4} = \{ \mathbf{x} \in \R^4 \colon A \mathbf{x} \leq \mathbf{b}, \mathbf{x} \geq \mathbf{0} \}$ where
\[ A = \begin{pmatrix} 7 & 4 & 1 & 0\\ 4 & 7 & 0 & 1 \\ 43 & 53 & 2 & 5 \\ 53 & 43 & 5 & 2 \end{pmatrix} \text{ and } \mathbf{b} = \begin{pmatrix} 1 \\ 1 \\ 8 \\ 8 \end{pmatrix}. \]

The polytope has $8$ facets in dimension $4$. Consider the linear program $\min \{ (1,1,1,1)^\top \mathbf{x} \colon \mathbf{x} \in M_{4} \}$. Todd showed that the shortest monotone path from the vertex $(1,1,8,8)/19$ to the optimum $\mathbf{0}$ of this LP is of length at least $5$, a contradiction to the monotone Hirsch conjecture \cite{ToddExample}.

We begin with a closer look at the graph of $M_{4}$. In the forty years since there has not been a detailed analysis of how many different orientations of $M_{4}$ have large monotone diameters. We will first address how rigid the selection of orientation is to obtain a counterexample to the monotone Hirsch conjecture. To perform this computation, we first enumerate all orientations of the graph induced by a linear objective function and then we compute the monotone diameter using a breadth first search. In \cite{edgotope}, the authors show that this set of orientations corresponds to the set of vertices of a zonotope they call the \emph{edgotope}. To compute this zonotope for a polytope $P$ with vertices $V(P)$ and edges $E(P) = \{(\mathbf{u},\mathbf{v}): \mathbf{u}, \mathbf{v} \in V(P), \mathbf{u} \text{ is adjacent to } \mathbf{v}\}$, one computes the following:
\[EZ(P) = \sum_{(\mathbf{u},\mathbf{v}) \in E(P)} \text{conv}(\{\mathbf{u},\mathbf{v}\}).\]
Zonotopes are dual to hyperplane arrangements, so this statement is equivalent to the observation that the set of orientations are in bijection with the set of regions of the hyperplane arrangement 
\[\mathcal{H} = \bigcup_{(\mathbf{u},\mathbf{v}) \in E(P)} \{ \mathbf{x} \colon (\mathbf{u}-\mathbf{v})^\top \mathbf{x} = 0\}.\] 
A region $R$ of $\mathcal{H}$ is uniquely determined by its sign vector $\mathbf{z} \in \{+,-\}^{E(P)}$ with entry $\mathbf{z}_{(\mathbf{u},\mathbf{v})}$ denoting whether $\mathbf{c}^\top(\mathbf{u} - \mathbf{v}) > 0$ or $\mathbf{c}^\top(\mathbf{u} - \mathbf{v}) < 0$ for each $(\mathbf{u},\mathbf{v}) \in E(P)$ and all $\mathbf{c} \in R$. Equivalently, the sign vector determines whether $\mathbf{c}^\top \mathbf{u} < \mathbf{c}^\top \mathbf{v}$ or $\mathbf{c}^\top \mathbf{v} > \mathbf{c}^\top \mathbf{u}$ for all $(\mathbf{u},\mathbf{v}) \in E(P)$ and $\mathbf{c} \in R$, which uniquely determines the orientation of the polytope.

To enumerate all the regions of the edgotope arrangement for $M_{4}$, we first found the graph $G(M_{4})$. This graph has precisely $40$ edges, which leads to an arrangement of $40$ hyperplanes. We used Sage \cite{sage} to compute the set of regions of this arrangement and found that there are exactly $7112$ regions and therefore $7112$ orientations. We enumerated the possible oriented graphs of $M_{4}$ for those orientations, and only five have diameters that contradict the monotone Hirsch conjecture. Figure \ref{fig:ToddGraph} shows the orientation given by Todd, and the remaining four orientations are in the Appendix. Five representatives of choices of $\mathbf{c}$ for which there is a bad orientation are $(1, 1, 1, 1),$ $(10716, 13680, 3477, 4465),$ $(13680, 10716, 4465, 3477),$ $(912, 1824, 513, 817),$  and $(1824, 912, 817, 513)$. These four orientations other than $(1,1,1,1)$ of the graph of $M_{4}$ are new and only differ from the Todd orientation on one edge in the first two cases and on two edges in the final two cases. 

Each of those orientations have $\mathbf{0}$ as the optimum and the only vertex of distance $5$ away is $(1,1,8,8)/19$. In all of the bad orientations, $(1,1,8,8)/19$ is not a maximizer of $\mathbf{c}$. It follows then that $M_{4}$ is not a counterexample to Ziegler's strict monotone Hirsch conjecture, which asks whether the Hirsch bound is satisfied for paths from maxima to minima across all orientations (see Chapter $3$ of \cite{z-95} for more details). Furthermore, there are $1832$ orientations for which $\mathbf{0}$ is the unique sink, so even among those oriented graphs, a large diameter is rare. 

While these observations are interesting on their own, in our context it allows us to reduce the set of orientations for which we need to prove there is always a monotone circuit walk down to those with $\mathbf{0}$ as a unique sink and coming from $(1,1,8,8)/19$. Note that the cones for this spindle are given by $\{ \mathbf{x} \in \R^{4} \colon \mathbf{x} \geq \mathbf{0} \}$ and $\{ \mathbf{x} \in \R^{4} \colon A \mathbf{x} \leq \mathbf{b} \}$. With these observations, we may prove our main result:

\begin{theorem} \label{thm:todd}
The monotone circuit diameter of $M_{4}$ is $4$. 
\end{theorem}

\begin{proof}
From our computations, for all orientations that do not have $\mathbf{0}$ as the unique optimum, there is always a monotone edge walk from any starting vertex of length at most $4$ and therefore always a monotone circuit walk of length at most $4$. Furthermore, the only case for which the shortest monotone edge walk is of length $5$ is when the starting vertex is $(1,1,8,8)/19$. Note that $\mathbf{0}$ is the apex of the cone $\{ \mathbf{x} \in \R^4 \colon \mathbf{x} \geq \mathbf{0} \}$ and $(1,1,8,8)/19$ is the apex of the cone $\{ \mathbf{x} \in \R^{4} \colon A \mathbf{x} \leq \mathbf{b} \}$.
Observe also that, since the entries of $-A$ are all non-positive, we have that $ \{ \mathbf{x} \in \R^{4} \colon \mathbf{x} \geq \mathbf{0}\} \subseteq \left\{ \mathbf{x} \in \R^{4} \colon -A  \mathbf{x} \leq \mathbf{0} \right\}.$ Hence, by \cref{thm:spindle-containment}, there must always exist a monotone circuit walk of length at most $4$ from $(1,1,8,8)/19$ to $\mathbf{0}$ for any orientation for which $\mathbf{0}$ is minimal. Therefore, the monotone circuit diameter of the Todd example $M_4$ is at most $4$. We may show computationally that it is exactly $4$ by verifying that $(1,1,8,8)/19$ is not a linear combination of any $3$ circuits. This implies that there is no circuit walk of length $3$ from $\mathbf{0}$ to $(1,1,8,8)/19$. \qed
\end{proof}

A natural question about our argument is how much it depends on the realization. Our proof does not apply to all possible realizations of the Todd polytope. It is, however, resilient under mild perturbations so long as the orientation stays fixed. Namely, fix a linear objective function $\mathbf{c}$ that induces an orientation of the graph of $M_{4}$ that contradicts the monotone Hirsch conjecture. Then, for a sufficiently small perturbation that preserves the containment of one cone within the other, $M_{4}$ together with $\mathbf{c}$ remains a counterexample to the monotone Hirsch conjecture while still having the same short circuit walk. 

For more general realizations and perturbations, there is more work to do. Since $M_{4}$ is simple, a sufficiently small perturbation does not change the graph of the polytope itself. Observe that, by our arguments thus far, the number of orientations of the Todd polytope graph is the number of regions of the hyperplane arrangement of linear hyperplanes with normal vectors given by the edge directions of $P$. By standard theory of hyperplane arrangements, the number of regions is determined by the linear matroid from the matrix with columns given by the normals of the hyperplanes in the arrangement. In particular, by the upper bound theorem for hyperplane arrangements, the number of regions is maximized precisely when the matroid is uniform. From an easy computation, the matroid generated by the edge directions of the Todd polytope is not uniform. However, after any sufficiently small generic perturbation, the matroid generated by the resulting edge directions will be uniform and hence, the number of orientations will increase. Furthermore, for a sufficiently small perturbation, the edge directions will remain approximately the same. This means that all orientations for the Todd polytope will remain, and only some new orientations are added. Thus, our arguments do not fully extend after perturbation if we only perturb the polytope, since new orientations may be created that we do not account for with our current arguments. 

\begin{figure}
    \[
\begin{tikzpicture}
    \draw[thick] (0, 0) node[red, circle,fill, inner sep = 1.5pt] {};
    \draw (0,.5) node[red] {\small $1234$};
    \draw[thick] (-1, 0) node[teal, circle,fill, inner sep = 1.5pt] {};
    \draw (-1,-.5) node[teal] {\small $1246$};
    \draw[thick] (-2, 1) node[orange, circle,fill, inner sep = 1.5pt] {};
    \draw (-2.75,1) node[orange] {\small $2346$};
    \draw[thick] (1, 0) node[teal, circle,fill, inner sep = 1.5pt] {};
    \draw (1,-.5) node[teal] {\small $1235$};
    \draw[thick] (2, 1) node[orange, circle,fill, inner sep = 1.5pt] {};
    \draw (2.75,1) node[orange] {\small $1345$};
    \draw[thick] (2, 3) node[teal, circle,fill, inner sep = 1.5pt] {};
    \draw (2,3.5) node[teal] {\small $3458$};
    \draw[thick] (3, 3) node[teal, circle,fill, inner sep = 1.5pt] {};
    \draw (3.75,3) node[teal] {\small $1458$};
    \draw[thick] (-3, 3) node[teal, circle,fill, inner sep = 1.5pt] {};
    \draw (-3.75,3) node[teal] {\small $2367$};
    \draw[thick] (-2, 3) node[teal, circle,fill, inner sep = 1.5pt] {};
    \draw (-2,3.5) node[teal] {\small $3467$};
    \draw[thick] (-4, 9) node[violet, circle,fill, inner sep = 1.5pt] {};
    \draw (-4,9.5) node[violet] {\small $2567$};
    \draw[thick] (4, 9) node[violet, circle,fill, inner sep = 1.5pt] {};
    \draw (4,9.5) node[violet] {\small $1568$};
    \draw[thick] (0, 3) node[teal, circle,fill, inner sep = 1.5pt] {};
    \draw (0,2.5) node[teal] {\small $3478$}; 
    \draw[thick] (-2, 5) node[green, circle,fill, inner sep = 1.5pt] {};
    \draw (-1.25,5.25) node[green] {\small $2358$}; 
    \draw[thick] (-3, 5) node[green, circle,fill, inner sep = 1.5pt] {};
    \draw (-3.15,5.5) node[green] {\small $2378$};
    \draw[thick] (3, 5) node[green, circle,fill, inner sep = 1.5pt] {};
    \draw (3.15,5.5) node[green] {\small $1478$};
    \draw[thick] (2, 5) node[green, circle,fill, inner sep = 1.5pt] {};
    \draw (1.25,5.25) node[green] {\small $1467$}; 
    \draw[thick] (-2, 7) node[violet, circle,fill, inner sep = 1.5pt] {};
    \draw (-1.25,7) node[violet] {\small $2578$}; 
    \draw[thick] (2, 7) node[violet, circle,fill, inner sep = 1.5pt] {};
    \draw (1.25,7) node[violet] {\small $1678$}; 
    \draw[thick] (0, 8) node[blue, circle,fill, inner sep = 1.5pt] {};
    \draw (.75,8) node[blue] {\small $\mathbf{5678}$};
    \draw[thick] (0, 10) node[green, circle,fill, inner sep = 1.5pt] {};
    \draw (0,10.5) node[green] {$1256$};
    \draw[thick, ->] (-1/10, 1/10) -- (-9/5, 4/5);
    \draw[thick, ->] (1/10, 1/10) -- (9/5, 4/5);
    \draw[thick, ->] (21/10, 5) -- (14/5, 5);
    \draw[thick, ->] (2, 51/10) -- (2, 34/5);
    \draw[thick, ->] (-3, 31/10) -- (-3, 24/5);
    \draw[thick, ->] (-29/10, 3) -- (-11/5, 3);
    \draw[thick, ->] (-9/10, 0) -- (-1/5, 0);
    \draw[thick, ->] (-9/10, 1/10) -- (9/5, 24/5);
    \draw[thick, ->] (-11/10, 1/10) -- (-9/5, 4/5);
    \draw[thick, ->] (-29/10, 51/10) -- (-11/5, 34/5);
    \draw[thick, ->] (29/10, 51/10) -- (11/5, 34/5);
    \draw[thick, ->] (-1/10, 31/10) -- (-14/5, 24/5);
    \draw[thick, ->] (1/10, 31/10) -- (14/5, 24/5);
    \draw[thick, ->] (-21/10, 11/10) -- (-14/5, 14/5);
    \draw[thick, ->] (-2, 11/10) -- (-2, 14/5);
    \draw[thick, ->] (19/10, 71/10) -- (1/5, 39/5);
    \draw[thick, ->] (-19/10, 31/10) -- (9/5, 24/5);
    \draw[thick, ->] (-19/10, 3) -- (-1/5, 3);
    \draw[thick, ->] (39/10, 89/10) -- (11/5, 36/5);
    \draw[thick, ->] (39/10, 89/10) -- (1/5, 41/5);
    \draw[thick, ->] (39/10, 89/10) -- (16/5, 16/5);
    \draw[thick, ->] (19/10, 3) -- (1/5, 3);
    \draw[thick, ->] (19/10, 31/10) -- (-9/5, 24/5);
    \draw[thick, ->] (3, 31/10) -- (3, 24/5);
    \draw[thick, ->] (29/10, 3) -- (11/5, 3);
    \draw[thick, ->] (-21/10, 5) -- (-14/5, 5);
    \draw[thick, ->] (-2, 51/10) -- (-2, 34/5);
    \draw[thick, ->] (-19/10, 71/10) -- (-1/5, 39/5);
    \draw[thick, ->] (-1/10, 99/10) -- (-4/5, 1/5);
    \draw[thick, ->] (1/10, 99/10) -- (19/5, 46/5);
    \draw[thick, ->] (-1/10, 99/10) -- (-19/5, 46/5);
    \draw[thick, ->] (1/10, 99/10) -- (4/5, 1/5);
    \draw[thick, ->] (-39/10, 89/10) -- (-16/5, 16/5);
    \draw[thick, ->] (-39/10, 89/10) -- (-1/5, 41/5);
    \draw[thick, ->] (-39/10, 89/10) -- (-11/5, 36/5);
    \draw[thick, ->] (9/10, 0) -- (1/5, 0);
    \draw[thick, ->] (9/10, 1/10) -- (-9/5, 24/5);
    \draw[thick, ->] (11/10, 1/10) -- (9/5, 4/5);
    \draw[thick, ->] (2, 11/10) -- (2, 14/5);
    \draw[thick, ->] (21/10, 11/10) -- (14/5, 14/5);
\end{tikzpicture}
\]
    \caption{The directed graph of $M_{4}$ for minimizing the linear function $(1,1,1,1)^\top \mathbf{x}$ used by Todd. The vertices are colored based on their distance to the optimum; there are $6$ different colors, as the diameter of the digraph is $5$. To match Figure $2$ in \cite{sy-15b}, we labeled the vertices by which inequalities are tight: inequalities $1$ to $4$ correspond to the rows of $A$, and $5$ to $8$ correspond to the non-negativity constraints.}
    \label{fig:ToddGraph}
\end{figure}
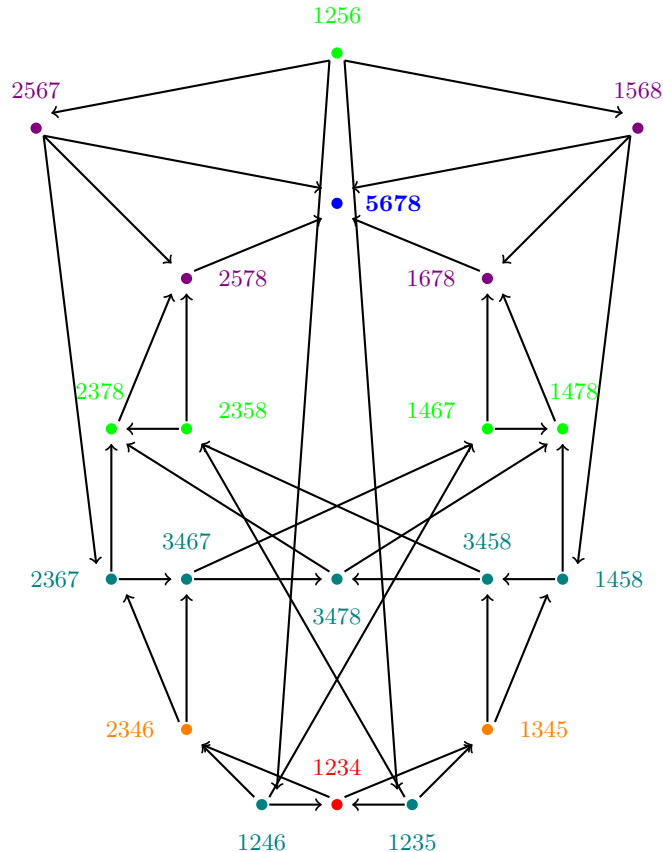

\subsection{Anti-Blocking Polytopes} \label{sec:antiblock}

While studying the Todd counterexample, we noticed that our arguments apply to a more general class of polytopes introduced by Fulkerson in \cite{AntiBlock} in the context of combinatorial optimization. Let $P$ be a polytope of the form $\{\mathbf{x} \in \mathbb{R}^d: A\mathbf{x} \leq \mathbf{b}, \mathbf{x} \geq \mathbf{0}\}$. Then we say that $P$ is {\em anti-blocking} or {\em down monotone} if whenever $\mathbf{x} \in P$, then also $\mathbf{x} - \mathbf{x}_{i} \mathbf{e}_{i} \in P$ for all $i \in [d]$. The following standard observation will be useful to leverage. We provide a brief proof.

\begin{lemma}
\label{lem:anti-block}
Let $P \subseteq \R^d$ be a full-dimensional polytope, and suppose that $P$ is anti-blocking. Then $\mathbf{0}$ is a vertex of $P$, and the feasible cone at $\mathbf{0}$ is precisely the non-negative orthant. 
\end{lemma}

\begin{proof}
Let $\mathbf{x} \in P$. Then $\mathbf{x} \geq \mathbf{0}$. Since $P$ is anti-blocking, $\mathbf{x} - \mathbf{x}_{1} \mathbf{e}_{1} \in P$, and by induction, $\mathbf{x} - \sum_{i=1}^{d} \mathbf{x}_{i} \mathbf{e}_{i} =  \mathbf{0} \in P$. Hence, $\mathbf{0} \in P$. Since $P \subseteq \{\mathbf{x} \colon \mathbf{x} \geq \mathbf{0}\}$, the feasible cone at $\mathbf{0}$ is contained in $\{\mathbf{x} \colon \mathbf{x} \geq \mathbf{0}\}$ and thus $\mathbf{0}$ must be a vertex of $P$. Since $P$ is full-dimensional and the set of vectors with distinct nonzero coordinates is dense in $\R^{d}$, there exists $\mathbf{x} \in P$ such that $\mathbf{x}_{i} > 0$ for all $i \in [d]$. However, then $\mathbf{x} - \sum_{i \in [d] \setminus \{j\}} \mathbf{x}_{i} \mathbf{e}_{i} = \mathbf{x}_{j} \mathbf{e}_{j} \in P$ for all $\mathbf{e}_{j}$. Thus, the feasible cone at $\mathbf{0}$ must contain $\mathbf{x}_{j} \mathbf{e}_{j}$ for all $j \in [d]$. Since the cone generated by $\mathbf{x}_{j} \mathbf{e}_{j}$ is the non-negative orthant, the feasible cone at $\mathbf{0}$ must therefore also be the non-negative orthant. \qed
\end{proof}

In the circuit setting, one can use this lemma to readily prove the existence of a circuit walk towards $\mathbf{0}$ that reduces the size of the support in each step.

\begin{theorem}\label{thm:antiblock}
Let $P \subseteq \R^{d}$ be a full-dimensional anti-blocking polytope. Then there is a circuit walk of length at most $d$ from any point $\mathbf{x} \in P$ to $\mathbf{0}$. Furthermore, this circuit walk is monotonically decreasing with respect to any linear objective function uniquely minimized at $\mathbf{0}$.
\end{theorem}

\begin{proof}
If $\mathbf{x} = \mathbf{0}$, there is nothing to prove. Thus, let $\mathbf{x} \in P$, and suppose that $\mathbf{x} \neq \mathbf{0}$. Note that there exists some $i \in [d]$ such that $\mathbf{x}_{i} > 0$. It suffices to see that there always is a monotone circuit step that reduces the size of support of $\mathbf{x}$ by one. 

By Lemma \ref{lem:anti-block}, the feasible cone of $P$ at $\mathbf{0}$ is the non-negative orthant. Hence, $\mathbf{e}_{i}$ is an edge direction of $P$, and thus also a circuit direction. This implies that a step from $\mathbf{x}$ to $\mathbf{x} - \mathbf{x}_{i} \mathbf{e}_{i}$ is in direction of a circuit. As $P$ is anti-blocking, $\mathbf{x} - \mathbf{x}_{i} \mathbf{e}_{i} \in P$, and since $P \subseteq \{\mathbf{x} \colon \mathbf{x} \geq \mathbf{0}\}$, such a step is of maximal step length. The feasible cone at $\mathbf{0}$ being the non-negative orthant further implies that a linear objective function $\mathbf{c} \in \R^d$ is uniquely minimized over $P$ at $\mathbf{0}$ if and only if $\mathbf{c} > \mathbf{0}$, i.e., if each component of $\mathbf{c}$ is strictly positive. Hence, $\mathbf{c}^\top(-\mathbf{x}_{i}\mathbf{e}_{i}) < 0$ which shows that the circuit step is also monotone. \qed
\end{proof}

As an immediate consequence, we have the following corollary for bounding circuit diameters of anti-blocking polytopes.

\begin{corollary}
The circuit diameter of a $d$-dimensional anti-blocking polytope is at most $d + \delta_{0}$, where $\delta_{0}$ is the largest length of a shortest circuit walk from $\mathbf{0}$ to any other vertex. 
\end{corollary}

Our final goal for this section is to exhibit that $M_4$ is, in fact, an anti-blocking polytope. To this end, we rely on the following well known characterization of anti-blocking for polytopes given in a certain description originally provided by Fulkerson in \cite{AntiBlock}.

\begin{lemma}\label{lem:charanti}
A full dimensional polytope $P$ with irredundant inequality description $P = \{  \mathbf{x}\colon A \mathbf{x} \leq \mathbf{b}, \mathbf{x} \geq \mathbf{0}\}$ for $\mathbf{b} \geq \mathbf{0}$ is anti-blocking if and only if $A \geq \mathbf{0}$. 
\end{lemma}

\begin{proof}
We prove the claim by contradiction. Suppose that there exists an entry $a_{ij}$ of $A$ 
such that $a_{ij} < 0$. Since the description of $P$ is irredundant, $F_{1} = \{\mathbf{x} \colon A_{i}^\top\mathbf{x} = \mathbf{b}_{i}\}$ and $F_{2} = \{\mathbf{x} \colon \mathbf{x}_{j} = 0\}$ are both facets of $P$ and therefore $F_{1}$ is not a proper subset of $F_{2}$. Hence, there exists $\mathbf{y} \in F_{1}$ such that $\mathbf{y} \notin F_{2}$. Then $\mathbf{y}_{j} > 0$, and $\sum_{j=1}^{n} a_{ij} \mathbf{y}_{j} = \mathbf{b}_{i}$. However, since $a_{ij} < 0$, 
\[A_{i}^\top(\mathbf{y} - \mathbf{y}_{j} \mathbf{e}_{j}) = -a_{ij}\mathbf{y}_{j} + \sum_{j=1}^{n} a_{ij} \mathbf{y}_{j} > \mathbf{b}_{i},\]
so $\mathbf{y} - \mathbf{y}_{j}\mathbf{e}_{j} \notin P$. Hence, $P$ is not anti-blocking. 
Suppose instead that $A \geq \mathbf{0}$. Then suppose that $\mathbf{x}$ satisfies $A\mathbf{x} \leq \mathbf{b}$ and $\mathbf{x} \geq \mathbf{0}$. Then, for all $i \in [d]$, $A(\mathbf{x}-\mathbf{x}_{i}\mathbf{e}_{i}) \leq A\mathbf{x} \leq \mathbf{b}$, and $\mathbf{x}-\mathbf{x}_{i}\mathbf{e}_{i} \geq \mathbf{0}$, so $\mathbf{x}-\mathbf{x}_{i} \mathbf{e}_{i} \in P$. Hence, $P$ is anti-blocking. \qed
\end{proof}

Recall the description of the Todd polytope $M_4$ at the beginning of Section \ref{sec:Todd}. Using Lemma \ref{lem:charanti}, it is easy to verify that $M_4$ is anti-blocking. By Theorem \ref{thm:antiblock}, there always exists a monotone circuit walk of at most $4$ steps to $\mathbf{0}$ for any objective function that is uniquely minimized at $\mathbf{0}$. We obtain the desired alternative argument that the Todd counterexample does not transfer to the circuit setting. 

\begin{corollary}
The Todd polytope $M_{4}$ is anti-blocking, and the LP $\min \{\mathbf{1}^\top\mathbf{x}: \mathbf{x} \in M_{4}\}$ is not a counterexample to the monotone circuit diameter conjecture.
\end{corollary}

%---------------------------------------------------------------------
\section{Bounded Hirsch Counterexamples}\label{sec:Santos}

In this section, we study the three $5$-dimensional spindles $S^{48}_5$, $S^{28}_5$, and $S^{25}_5$ that provide the basis of Santos' construction for obtaining bounded Hirsch counterexamples. The key property of these spindles is that their length of $6$ strictly exceeds their dimension $5$. In contrast, we prove in \cref{sec:Santos-5} that their circuit length is at most $5$.  

Our arguments rely on a careful analysis of certain $2$-faces and, as we will see, extend to all realizations of the spindles that satisfy mild assumptions, including slight perturbations. As an additional benefit, the arguments also enable us in \cref{sec:Santos-highdim} to verify directly that neither of the two explicit (high-dimensional) Hirsch counterexamples from \cite{msw-15}, which are specific instances of Santos' construction, has a circuit length exceeding the dimension.

\subsection{The Circuit Length of the $5$-Dimensional Spindles} \label{sec:Santos-5}

We begin with a key observation on circuit walks on $2$-dimensional polytopes: under certain conditions, one can guarantee that two circuit steps suffice to reach a vertex from a specific set of target vertices. %a specific target vertex from any other vertex.

\begin{lemma} \label{lem:2-face}
Let $P = \{ \mathbf{x} \in \R^d \colon A \mathbf{x} = \mathbf{b}, B \mathbf{x} \le \mathbf{d} \}$ be a $2$-dimensional polytope, and let $\mathcal{V}$ be a subset of its vertices. Further, let $C(P,\mathcal{V}) = \{ \mathbf{x} \in \R^d \colon A \mathbf{x} = \mathbf{b}, B' \mathbf{x} \le \mathbf{d}' \}$, where the system $B' \mathbf{x} \le \mathbf{d}'$ consists of all inequalities from $B \mathbf{x} \le \mathbf{d}$ that are not tight at any vertex $\mathbf{v} \in \mathcal{V}$.
If $C(P,\mathcal{V})$ is unbounded, then for all vertices $\mathbf{y}_0 \notin \mathcal{V}$ of $P$ there is a circuit walk of length at most $2$ from $\mathbf{y}_0$ to some $\mathbf{v} \in \mathcal{V}$ on $P$.
\end{lemma} 
\begin{proof}

First, note that any vertex $\mathbf{v} \in \mathcal{V}$ can be reached via a single circuit step starting from any point on an edge of $P$ incident with $\mathbf{v}$. In particular, this is true for the two adjacent vertices of $\mathbf{v}$. Thus, if every vertex of $P$ is adjacent to a vertex in $\mathcal{V}$, then we are done, so suppose there exists a vertex $\mathbf{y}_0$ of $P$ that is not adjacent to any vertex in $\mathcal{V}$. It suffices to show that there is a point $\mathbf{y}_1$ on an edge of $P$ incident with some vertex $\mathbf{v} \in \mathcal{V}$ that can be reached from $\mathbf{y}_0$ in at most one circuit step.

Note that, because $\mathbf{y}_{0}$ is not adjacent to any vertex in $\mathcal{V}$, the inequalities tight at $\mathbf{y}_{0}$ are not tight for any vertex in $\mathcal{V}$. Hence, $C(P, \mathcal{V})$ is contained in the feasible cone at $\mathbf{y}_{0}$ and is therefore contained in a proper pointed cone. It follows that, since $C(P, \mathcal{V})$ is also unbounded by hypothesis, its recession cone contains an extreme ray generated by some $\mathbf{g}\in \R^d$. Note that $\mathbf{g}$ is a circuit of $P$, because the recession cone of $C(P,\mathcal{V})$ is a union of cones of the elementary arrangement of $P$.
.
We further have that $\mathbf{y}_0 \in C(P,\mathcal{V})$ since $P \subset C(P,\mathcal{V})$. 
It follows that the ray given by $\mathbf{y}_0 + \mu \mathbf{g}$ for all $\mu \ge 0$ is contained in $C(P,\mathcal{V})$ and follows a circuit direction. Since $P$ is bounded but $C(P,\mathcal{V})$ is not, this ray intersects an edge of $P$ incident with some vertex $\mathbf{v} \in \mathcal{V}$ in a point $\mathbf{y}_1$. Thus, $\mathbf{y}_1$ can be reached from $\mathbf{y}_0$ in a single circuit step. \qed 

\end{proof}

\begin{figure}[hbt]
\centering
\begin{tikzpicture}[every node/.style={inner sep=5pt}]
    \coordinate (y0) at (0,0);
    \coordinate (v) at (4,.25);
    \draw[thick] (0,0) -- (.5,.75) -- (3,1.75) -- (v) -- (2.5,-1.5) -- (.5,-1) -- cycle;
    \draw[thick,dashed,-stealth] (3,1.75) -- (4.25,2.25);
    \draw[thick,dashed,-stealth] (2.5,-1.5) -- (4,-1.875);
    \node (p) at (3.75,-1.35) {$C(P,\{\mathbf{v}\})$};
    \node (p) at (1,-.75) {$P$};

    \coordinate (y1) at (125/38,50/38); % ($(y0) + (2.5,1)$);
    \draw[black,fill] (v) circle (2pt) node[anchor=west] {$\mathbf{v}$};
    \draw[red,fill] (y0) circle (2pt) node[anchor=east] {$\mathbf{y}_0$};
    \draw[red,fill] (y1) circle (2pt) node[anchor=west] {$\mathbf{y}_1$};
    \draw[very thick,red,-stealth] (y0) -- (y1);
    \draw[very thick,red,-stealth] (y1) -- (v);

\end{tikzpicture}
\caption{A polytope $P$ with a vertex $\mathbf{v}$ and the corresponding polyhedron $C(P,\{\mathbf{v}\})$ as defined in \cref{lem:2-face}. The recession cone of $C(P,\{\mathbf{v}\})$ is spanned by the two dashed edge directions. Starting at the vertex $\mathbf{y}_0$, we can follow any of them to a point $\mathbf{y}_1$ on one of the two edges of $P$ incident with $\mathbf{v}$. This gives a circuit walk on $P$ from $\mathbf{y}_0$ to $\mathbf{v}$ of length $2$.}
\label{fig:2-face-example}
\end{figure}
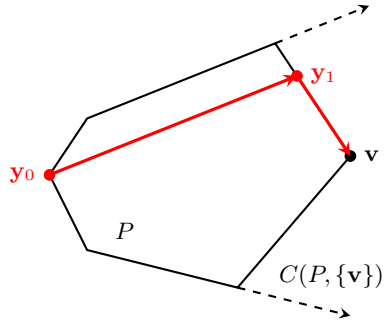

Informally, \cref{lem:2-face} guarantees that for any starting vertex $\mathbf{y}_0$ of a $2$-dimensional polytope $P$, there is always a circuit step to an edge incident with some vertex in $\mathcal{V}$ if the omission of all inequalities that are tight at any of the vertices in $\mathcal{V}$ would make $P$ unbounded. While the statement and proof are phrased in terms of a vertex $\mathbf{y_0}$, the statement remains true for any point in $P$.

For each of the three spindles, we will now show the following: within $3$ edge steps from one apex, one can reach a $2$-face that contains the other apex and satisfies the condition of \cref{lem:2-face} if we choose the other apex for the vertex $\mathbf{v}$. We first explain the arguments in detail for the $5$-dimensional spindle $S_5^{48}$ with $48$ facets and length $6$ that Santos' original counterexample in \cite{s-11} is constructed from. Following \cite[Theorem 3.1]{s-11}, it is given by the description $S_5^{48} = \{ \mathbf{x} \in \R^5 \colon A^+ \mathbf{x} \le \mathbf{1}, A^- \mathbf{x} \le \mathbf{1} \}$ where $\mathbf{1}$ denotes the all-one vector and $A^+$ and $A^-$ are the matrices
\[ \makeatletter\setlength\BA@colsep{4pt}\makeatother
  A^+ =\; \begin{blockarray}{cccccc}
  \begin{block}{(ccccc)r}
  1 & \pm 18 & 0 & 0 & 0 & \;\mathsf{ 1^\pm} \\
  1 & 0 & \pm 18 & 0 & 0 & \;\mathsf{ 2^\pm} \\
  1 & 0 & 0 & \pm 45 & 0 & \;\mathsf{ 3^\pm} \\
  1 & 0 & 0 & 0 & \pm 45 & \;\mathsf{ 4^\pm} \\
  1 & \pm 15 & 15 & 0 & 0 & \;\mathsf{ 5^\pm} \\
  1 & \pm 15 & -15 & 0 & 0 & \;\mathsf{ 6^\pm} \\
  1 & 0 & 0 & \pm 30 & 30 & \;\mathsf{ 7^\pm} \\
  1 & 0 & 0 & \pm 30 & -30 & \;\mathsf{ 8^\pm} \\
  1 & 0 & \pm 10 & 40 & 0 & \;\mathsf{ 9^\pm} \\
  1 & 0 & \pm 10 & -40 & 0 & \;\mathsf{10^\pm} \\
  1 & \pm 10 & 0 & 0 & 40 & \;\mathsf{11^\pm} \\
  1 & \pm 10 & 0 & 0 & -40 & \;\mathsf{12^\pm} \\
  \end{block}
  \end{blockarray},\;
  %
  %\text{ and }
  %
  A^- =\; \begin{blockarray}{cccccc}
  \begin{block}{(ccccc)r}
  -1 & 0 & 0 & 0 & \pm 18 & \;\mathsf{13^\pm} \\
  -1 & 0 & 0 & \pm 18 & 0 & \;\mathsf{14^\pm} \\
  -1 & \pm 45 & 0 & 0 & 0 & \;\mathsf{15^\pm} \\
  -1 & 0 & \pm 45 & 0 & 0 & \;\mathsf{16^\pm} \\
  -1 & 0 & 0 & 15 & \pm 15 & \;\mathsf{17^\pm} \\
  -1 & 0 & 0 & -15 & \pm 15 & \;\mathsf{18^\pm} \\
  -1 & \pm 30 & 30 & 0 & 0 & \;\mathsf{19^\pm} \\
  -1 & \pm 30 & -30 & 0 & 0 & \;\mathsf{20^\pm} \\
  -1 & 40 & 0 & \pm 10 & 0 & \;\mathsf{21^\pm} \\
  -1 & -40 & 0 & \pm 10 & 0 & \;\mathsf{22^\pm} \\
  -1 & 0 & 40 & 0 & \pm 10 & \;\mathsf{23^\pm} \\
  -1 & 0 & -40 & 0 & \pm 10 & \;\mathsf{24^\pm} \\
  \end{block}
  \end{blockarray}
\]

with $24$ rows each, labeled $\mathsf{1^\pm}$ to $\mathsf{24^\pm}$. We will also use these labels for the corresponding inequalities and the facets that they define.  
The two apices of $S_5^{48}$ are $\mathbf{v}^+ = (1,0,0,0,0)$ and $\mathbf{v}^- = (-1,0,0,0,0)$. The spindle $S_5^{48}$ can equivalently be written as $S_5^{48} = (C^+ + \mathbf{v}^+) \cap (C^- + \mathbf{v}^-)$ for the two cones $C^+ = \{ \mathbf{x} \in \R^5 \colon A^+ \mathbf{x} \le \mathbf{0} \}$ and $C^- = \{ \mathbf{x} \in \R^5 \colon A^- \mathbf{x} \le \mathbf{0} \}$.

Among the properties of $S_5^{48}$ listed in \cite{s-11} are the following two symmetries. They are direct consequences of the symmetry in the coefficients of $A^\pm$.
\begin{proposition}[\cite{s-11}] \label{prop:santos-symmetry}
The following linear transformations of $\R^5$ leave $S_5^{48}$ invariant:
\begin{myenumerate}
    \item $(x_1,x_2,x_3,x_4,x_5) \mapsto (-x_1,x_5,x_4,x_2,x_3)$, \label{enum:santos-symmetry-i}
    \item $(x_1,x_2,x_3,x_4,x_5) \mapsto (x_1,x_3,x_2,x_5,x_4)$.  \label{enum:santos-symmetry-ii}
\end{myenumerate}
\end{proposition}

We note that the transformation given in \cref{prop:santos-symmetry}\cref{enum:santos-symmetry-i} switches the roles of $A^+$ and $A^-$ and of $v^+$ and $v^-$, while \cref{enum:santos-symmetry-ii} is a permutation of coordinates that preserves the matrices $A^\pm$ up to reordering rows.
With these observations we are now able to bound the circuit length of $S_5^{48}$.

\begin{theorem} \label{thm:santos-length}
The circuit length of $S_5^{48}$ is at most $5$.
\end{theorem}
\begin{proof}
By \cref{prop:santos-symmetry}\cref{enum:santos-symmetry-i}, $S_5^{48}$ is symmetric under a linear transformation that sends $\mathbf{v}^+$ to $\mathbf{v}^-$ and vice versa. Hence, from any circuit walk from $\mathbf{v}^+$ to $\mathbf{v}^-$ on $S_5^{48}$, one can immediately obtain a circuit walk from $\mathbf{v}^-$ to $\mathbf{v}^+$ by applying the same linear transformation.
It therefore suffices to find a circuit walk of length at most $5$ in one of the directions, say from $\mathbf{v}^+$ to $\mathbf{v}^-$, to prove the statement. 

Let $F$ be the $2$-face of $S_5^{48}$ defined by inequalities $\mathsf{15}^+$, $\mathsf{19}^+$, and $\mathsf{21}^+$. \Cref{fig:santos-face} shows the graph of $F$, along with shortest paths from $\mathbf{v}^+$ to four vertices of $F$ that are at distance $3$ from $\mathbf{v}^+$ in the graph of $S_5^{48}$ (graph computations were done using Polymake \cite{polymake}). For this face $F$, we now claim that the polyhedron $C(F,\{\mathbf{v}^-\})$ as defined in \cref{lem:2-face} is unbounded. It then follows from \cref{lem:2-face} that from either of the four vertices that are at distance $3$ from $\mathbf{v}^+$, there is a circuit walk on $F$ to the other apex $\mathbf{v^-}$ of length at most $2$. Since circuits of $F$ are also circuits of $S_5^{48}$, this means that there is a circuit walk on $S_5^{48}$ of length at most $5$ from $\mathbf{v}^+$ to $\mathbf{v}^-$.

To verify that $C(F,\{\mathbf{v}^-\})$ is unbounded, it suffices to show the following property:

\begin{claim}[$\ast$] \label{claim:bounded}%\tag{$\ast$}
The polyhedron $Q$ given by inequalities $\mathsf{13^\pm}$ to $\mathsf{24^\pm}$ together with $\mathsf{4^\pm}$ is bounded.
\end{claim}

Before we give a proof, let us first argue why the claim implies unboundedness of $C(F,\{\mathbf{v}^-\})$.

The inequalities $\mathsf{13^\pm}$ to $\mathsf{24^\pm}$ and $\mathsf{4^\pm}$ that determine $Q$ must also be facet-defining for $Q$ since each of them defines a facet of $S_5^{48}$ and $S_5^{48} \subset Q$.
Let $F_Q$ be the $2$-face of $Q$ defined by $\mathsf{15^+}$, $\mathsf{19^+}$, and $\mathsf{21^+}$. 
Since $S_5^{48} \subset Q$, we also have that $F \subset F_Q$. Hence, $\mathbf{v}^-$ and its two adjacent vertices $\mathsf{4^+13^+17^+23^+}$ and $\mathsf{4^-13^-17^-23^-}$ in \cref{fig:santos-face} are also vertices of $F_Q$. Moreover, any edge of $F_Q$ that does not contain $\mathbf{v}^-$ can only be defined by $\mathsf{4^+}$ or $\mathsf{4^-}$ because $\mathbf{v}^-$ is in all other facets of $Q$. 
In fact, both $\mathsf{4}^+$ and $\mathsf{4}^-$ define edges of $F_Q$. To see this, note that the (proper) face of $F_Q$ defined by $\mathsf{4}^+$ contains an edge of $F$ (namely, the edge between vertices $\mathsf{4^+13^+17^+23^+}$ and $\mathsf{4^+11^+}$ in \cref{fig:santos-face}) and, hence, must be an edge of $F_Q$ itself. A similar argument shows that $\mathsf{4^-}$ defines an edge of $F_Q$, too, which is incident with the vertex $\mathsf{4^-13^-17^-23^-}$. These two edges of $F_Q$ must be distinct as the two vertices $\mathsf{4^+13^+17^+23^+}$ and $\mathsf{4^-13^-17^-23^-}$ would otherwise be adjacent on $F_Q$ and therefore also on $F$; a contradiction (see \cref{fig:santos-face}).

Now suppose that $Q$ is bounded. Then $F_Q$ is bounded, which means that the two edges of $F_Q$ defined by $\mathsf{4^+}$ and $\mathsf{4^-}$ must intersect in a vertex $\mathbf{w}$ of $F_Q$. Further observe that for each vertex of $F$ in \cref{fig:santos-face}, at most one of the two inequalities $\mathsf{4^\pm}$ is tight. This means that $\mathbf{w}$ cannot be a vertex of $F$. Thus, $F_Q$ has exactly four vertices, namely $\mathbf{v}^-$, $\mathsf{4^+13^+17^+23^+}$, $\mathsf{4^-13^-17^-23^-}$, and $\mathbf{w}$. The first three are also vertices of $F$ and therefore satisfy all inequalities that define edges of $F$. Since $F_Q$ is bounded and $F \subset F_Q$, any inequality that defines an edge of $F$ but not of $F_Q$ must cut off a vertex of $F_Q$, and that can only be $\mathbf{w}$. Since each edge of $C(F,\{\mathbf{v}^-\})$ is either defined by $\mathsf{4^\pm}$ or by an inequality that defines an edge of $F$ but not of $F_Q$, the recession cone of $C(F,\{\mathbf{v}^-\})$ is therefore precisely the $2$-dimensional feasible cone of $F_Q$ at $\mathbf{w}$. This means that $C(F,\{\mathbf{v}^-\})$ is unbounded as desired.\\

It remains to prove Claim ($\ast$). By definition of $Q$, the recession cone of $Q$ consists of all vectors $\mathbf{x} \in C^-$ that further satisfy the two facet-defining inequalities $\mathsf{4^\pm}$ for the cone $C^+$ (with right-hand side $0$). Since the two rows $\mathsf{4^\pm}$ of $A^+$ add up to $(2,0,0,0,0)$, all vectors $\mathbf{x}$ in the recession cone of $Q$ must satisfy $x_1 \le 0$. Further, the sum of all rows $\mathsf{13^\pm}$ to $\mathsf{24^\pm}$ of $A^-$ equals $(-24,0,0,0,0)$. This implies that $(-1,0,0,0,0)$ is in the strict interior of the polar cone of $C^-$. For the linear objective function $(-1,0,0,0,0)^\top \mathbf{x}$, the origin $\mathbf{0}$ is therefore the unique maximizer over $C^-$. Hence, any nonzero vector $\mathbf{x} \in C^-$ must satisfy $x_1>0$. The recession cone of $Q$ therefore only contains $\mathbf{0}$, which means that $Q$ must be bounded. \qed
\end{proof}

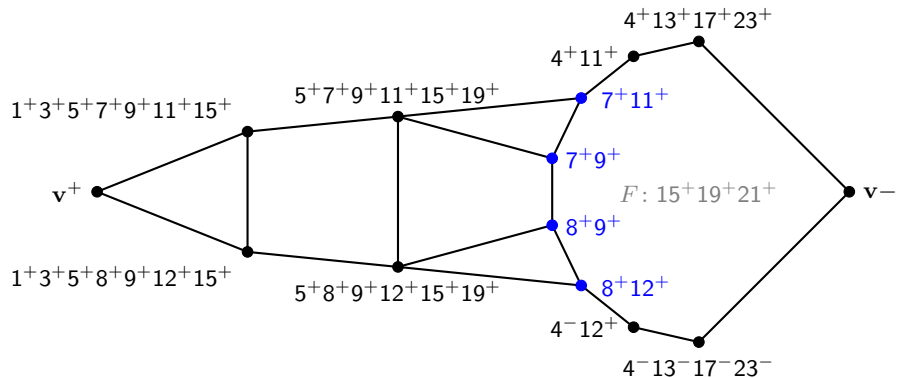
\begin{figure}[hbt]
\centering
\begin{tikzpicture}[every node/.style={inner sep=5pt}]
    \coordinate (v155) at (-1.2246467991473532e-16, -2.0);
    \coordinate (v69) at (-0.8677674782351164, -1.8019377358048383);
    \coordinate (v65) at (-1.5636629649360596, -1.246979603717467);
    \coordinate (v56) at (-1.9498558243636472, -0.4450418679126288);
    \coordinate (v68) at (-1.9498558243636472, 0.4450418679126288);
    \coordinate (v53) at (-1.5636629649360596, 1.246979603717467);
    \coordinate (v54) at (-0.8677674782351164, 1.8019377358048383);
    \coordinate (v98) at (-1.2246467991473532e-16, 2.0);
    \coordinate (v0) at (2, 0);
    \coordinate (v255) at (-6,.8);
    \coordinate (v99) at (-4,1);
    \coordinate (v29) at (-6,-.8);
    \coordinate (v46) at (-4,-1);
    \coordinate (v132) at (-8,0);
    
    \draw[thick] (v0) -- (v98);
    \draw[thick] (v0) -- (v155);
    \draw[thick] (v65) -- (v46);
    \draw[thick] (v65) -- (v56);
    \draw[thick] (v65) -- (v69);
    \draw[thick] (v98) -- (v54);
    \draw[thick] (v99) -- (v255);
    \draw[thick] (v99) -- (v46);
    \draw[thick] (v99) -- (v68);
    \draw[thick] (v99) -- (v53);
    \draw[thick] (v132) -- (v255);
    \draw[thick] (v132) -- (v29);
    %\draw[thick] (v37) -- (v29);
    %\draw[thick] (v37) -- (v41);
    %\draw[thick] (v37) -- (v56);
    %\draw[thick] (v68) -- (v41);
    \draw[thick] (v68) -- (v53);
    \draw[thick] (v68) -- (v56);
    \draw[thick] (v69) -- (v155);
    %\draw[thick] (v41) -- (v255);
    \draw[thick] (v46) -- (v29);
    \draw[thick] (v46) -- (v56);
    \draw[thick] (v53) -- (v54);
    \draw[thick] (v29) -- (v255);

    \draw[black,fill] (v54) circle (2pt) node[anchor=east] {$\mathsf{4^+11^+}$};
    \draw[black,fill] (v98) circle (2pt) node[anchor=south] {$\mathsf{4^+13^+17^+23^+}$};
    \draw[black,fill] (v0) circle (2pt) node[anchor=west] {$\mathbf{v}-$};
    \draw[black,fill] (v155) circle (2pt) node[anchor=north] {$\mathsf{4^-13^-17^-23^-}$};
    \draw[black,fill] (v69) circle (2pt) node[anchor=east] {$\mathsf{4^-12^+}$};
    \draw[blue,fill] (v65) circle (2pt) node[anchor=west,inner sep=7.5pt] {$\mathsf{8^+12^+}$};
    \draw[blue,fill] (v56) circle (2pt) node[anchor=west] {$\mathsf{8^+9^+}$};
    \draw[blue,fill] (v68) circle (2pt) node[anchor=west] {$\mathsf{7^+9^+}$};
    \draw[blue,fill] (v53) circle (2pt) node[anchor=west,inner sep=7.5pt] {$\mathsf{7^+11^+}$};
    \draw[black,fill] (v132) circle (2pt) node[anchor=east] {$\mathbf{v}^+$};
    \draw[black,fill] (v255) circle (2pt) node[anchor=south east] {$\mathsf{1^+3^+5^+7^+9^+11^+15^+}$};
    \draw[black,fill] (v99) circle (2pt) node[anchor=south] {$\mathsf{5^+7^+9^+11^+15^+19^+}$};
    \draw[black,fill] (v29) circle (2pt) node[anchor=north east] {$\mathsf{1^+3^+5^+8^+9^+12^+15^+}$};
    \draw[black,fill] (v46) circle (2pt) node[anchor=north] {$\mathsf{5^+8^+9^+12^+15^+19^+}$};

    \node[gray] at (0,0) {$F \colon \mathsf{15^+19^+21^+}$};
\end{tikzpicture}
\caption{The subgraph of $S_5^{48}$ induced by all vertices of the $2$-face $F$ defined by inequalities $\mathsf{15^+}$, $\mathsf{19^+}$, and $\mathsf{21^+}$, and by vertices on a shortest path from $\mathbf{v}^+$ to a vertex of $F$. Vertex labels indicate which inequalities are tight. For the vertices of $F$, we omitted the facets containing $F$ from their labels. The four highlighted vertices are at distance $3$ from $\mathbf{v}^+$. 
}
\label{fig:santos-face}
\end{figure}

Our proof of \cref{thm:santos-length} exhibits that it is possible to verify that the circuit length of $S_5^{48}$ is strictly less than the combinatorial length through the geometry of its $2$-faces: From each of the vertices $\mathsf{7^+11^+15^+19^+21^+}$ and $\mathsf{8^+12^+15^+19^+21^+}$ in \cref{fig:santos-face}, there is an edge walk of length $3$ to $\mathbf{v}^-$ that stays on the $2$-face $F$. 
In contrast, two circuit steps suffice to reach $\mathbf{v}^-$ as shown above.

We further note that Claim ($\ast$) is the only geometric property of $S_5^{48}$ used in the proof of \cref{thm:santos-length}. The remainder of the argument does not depend on the particular realization. 
Indeed, consider a different realization $\widetilde{S}_5^{48}$ whose facets (and facet-defining inequalities) are labeled $\mathsf{1}^\pm$ to $\mathsf{24}^\pm$ again such that facets of $\widetilde{S}_5^{48}$ and $S_5^{48}$ with the same label are combinatorially equivalent. Suppose that $\widetilde{S}_5^{48}$ satisfies Claim ($\ast$). As an immediate consequence of the proof of \cref{thm:santos-length}, there is a circuit walk of length at most $5$ from one of the apices of $\widetilde{S}_5^{48}$ to the other one. However, since $\widetilde{S}_5^{48}$ might not be symmetric under the linear transformation in \cref{prop:santos-symmetry}\cref{enum:santos-symmetry-i}, this does not directly imply the existence of a short walk in the converse direction as well. In order to extend our bound on the circuit length to $\widetilde{S}_5^{48}$, we require validity of a `symmetric' version of Claim ($\ast$). More precisely, we obtain the following corollary. 

\begin{corollary} \label{cor:santos-length-realization}
Let $\widetilde{S}_5^{48}$ be a realization of $S_5^{48}$ with facets and facet-defining inequalities labeled $\mathsf{1}^\pm$ to $\mathsf{24}^\pm$ such that each facet is combinatorially equivalent to the facet of $S_5^{48}$ with the same label. Suppose that $\widetilde{S}_5^{48}$ satisfies both of the following properties:
\begin{myenumerate}
    \item The polyhedron given by inequalities $\mathsf{13^\pm}$ to $\mathsf{24^\pm}$ together with one of the pairs $\mathsf{3^\pm}$ or $\mathsf{4^\pm}$ is bounded. \label{enum:santos-length-realization-i}
    \item The polyhedron given by inequalities $\mathsf{1^\pm}$ to $\mathsf{12^\pm}$ together with one of the pairs $\mathsf{15^\pm}$ or $\mathsf{16^\pm}$ is bounded. \label{enum:santos-length-realization-ii}
\end{myenumerate}
Then the circuit length of $\widetilde{S}_5^{48}$ is at most $5$. 
\end{corollary}

In particular, \cref{cor:santos-length-realization} applies to all realizations resulting from mild perturbations of $S_5^{48}$.

\begin{proof}
Under the permutation of coordinates given in \cref{prop:santos-symmetry}\cref{enum:santos-symmetry-ii} (which is an involution and leaves $S_5^{48}$ invariant) the pairs of inequalities $\mathsf{3}^\pm$ and $\mathsf{4^\pm}$, and $\mathsf{15^\pm}$ and $\mathsf{16^\pm}$ are in correspondence. This means that $S_5^{48}$ has a $2$-face $G$ that is linearly isomorphic to the $2$-face $F$ from the proof of \cref{thm:santos-length} and satisfies Claim ($\ast$) with $\mathsf{3^\pm}$ instead of $\mathsf{4^\pm}$. Since any realization of $S_5^{48}$ has two faces that are combinatorially equivalent with $F$ and $G$, respectively, we conclude from the proof of \cref{thm:santos-length} that there is a circuit walk of length at most $5$ from the apex not contained in $F$ or $G$ to the other one on any realization with property \cref{enum:santos-length-realization-i}.

For the converse direction, recall that $S^{48}_5$ is also invariant under the linear transformation given in \cref{prop:santos-symmetry}\cref{enum:santos-symmetry-i}. For the pairs of facets defined by inequalities $\mathsf{3^\pm}$ and $\mathsf{4^\pm}$, the corresponding facets in the image are defined by $\mathsf{15^\pm}$ and $\mathsf{16^\pm}$. Thus, if we assume property \cref{enum:santos-length-realization-ii}, the existence of a circuit walk of length at most $5$ in the converse direction follows directly from the proof of \cref{thm:santos-length} again. \qed
\end{proof}

For the spindle $S_5^{48}$ as realized in \cite[Theorem 3.1]{s-11}, we are even able to determine its exact circuit length.

\begin{theorem} \label{cor:santos-width}
The circuit length of $S^{48}_5$ is $2$.
\end{theorem}
\begin{proof}
Consider again the face $F$ of $S_5^{48}$ defined by inequalities $\mathsf{15^+}$, $\mathsf{19^+}$, and $\mathsf{21^+}$ from the proof of \cref{thm:santos-length}. Let $\mathbf{y} = (0, \frac{1}{45}, \frac{1}{90}, \frac{1}{90}, \frac{7}{360}) = \frac{5}{8} \cdot \mathsf{7^+11^+15^+19^+21^+} + \frac{3}{8} \cdot \mathsf{4^+11^+15^+19^+21^+}$ (where we use the vertex labeling of \cref{fig:santos-face} to refer to the corresponding coordinate vectors). By this construction, $\mathbf{y}$ is a point in the strict interior of the edge of $F$ defined by $\mathsf{11^+}$. Moreover, each of the two vectors $360 (\mathbf{y}-\mathbf{v}^\pm) = (\pm 360,8,4,4,7)$ is a circuit: $(360,8,4,4,7)$ is parallel to the four facets $\mathsf{12^-}$, $\mathsf{15^+}$, $\mathsf{19^+}$, and $\mathsf{21^+}$; and $(-360,8,4,4,7)$ is parallel to $\mathsf{11^+}$, $\mathsf{15^-}$, $\mathsf{20^-}$, and $\mathsf{22^-}$. In both cases, it can be verified by a direct computation that the corresponding rows of $A^\pm$ are linearly independent.
Since the edge of $F$ defined by $\mathsf{11^+}$ neither contains $\mathbf{v}^+$ nor $\mathbf{v}^-$, the point $\mathbf{y}$ can be reached from either of $\mathbf{v}^\pm$ via a circuit step of maximal step length each. We conclude that both the sequence $\mathbf{v}^+, \mathbf{y}, \mathbf{v}^-$ and its reverse are circuit walks of length $2$.

Further, the vector $(\mathbf{v}^+ - \mathbf{v}^-)/2 = (1,0,0,0,0)$ is not contained in any facet of $S_5^{48}$
%$C^+$ or $C^-$ 
and thus cannot be a circuit. Hence, the circuit length of $S_{5}^{48}$ is $2$. \qed 
\end{proof}

Santos' original example constructed from $S_5^{48}$ is not the lowest-dimensional bounded Hirsch counterexample known to date. In \cite{msw-15}, Matschke, Santos, and Weibel gave two smaller counterexamples, both of which are constructed from 5-dimensional spindles of length 6 with $28$ and $25$ facets, respectively. The first one, $S^{28}_5$, from \cite[Corollary 2.9]{msw-15} is given by $S_5^{28} = \{ \mathbf{x} \in \R^5 \colon A^+ \mathbf{x} \le \mathbf{1}, A^- \mathbf{x} \le \mathbf{1} \}$ for the matrices
\[ \makeatletter\setlength\BA@colsep{4pt}\makeatother
  A^+ =\; \begin{blockarray}{cccccc}
  \begin{block}{(ccccc)l}
  1 & \pm 18 & 0 & 0 & 0 & \;\mathsf{ 1^\pm} \\
  1 & 0 & 0 & \pm 30 & 0 & \;\mathsf{ 2^\pm} \\
  1 & 0 & 0 & 0 & \pm 30 & \;\mathsf{ 3^\pm} \\
  1 & 0 & 5 & 0 & \pm 25 & \;\mathsf{ 4^\pm} \\
  1 & 0 & -5 & 0 & \pm 25 & \;\mathsf{ 5^\pm} \\
  1 & 0 & 0 & 18 & \pm 18 & \;\mathsf{ 6^\pm} \\
  1 & 0 & 0 & -18 & \pm 18 & \;\mathsf{ 7^\pm} \\
  \end{block}
  \end{blockarray},\;
  %
  %\text{ and }
  %
  A^- =\; \begin{blockarray}{cccccc}
  \begin{block}{(ccccc)l}
  -1 & 0 & 0 & \pm 18 & 0 & \;\mathsf{ 8^\pm} \\
  -1 & 0 & \pm 30 & 0 & 0 & \;\mathsf{ 9^\pm} \\
  -1 & \pm 30 & 0 & 0 & 0 & \;\mathsf{10^\pm} \\
  -1 & 25 & 0 & 0 & \pm 5 & \;\mathsf{11^\pm} \\
  -1 & -25 & 0 & 0 & \pm 5 & \;\mathsf{12^\pm} \\
  -1 & 18 & \pm 18 & 0 & 0 & \;\mathsf{13^\pm} \\
  -1 & -18 & \pm 18 & 0 & 0 & \;\mathsf{14^\pm} \\
  \end{block}
  \end{blockarray}
\]
with $14$ rows each, labeled $\mathsf{1^\pm}$ to $\mathsf{14^\pm}$. Again, the two apices are $\mathbf{v}^+ = (1,0,0,0,0)$ and $\mathbf{v}^- = (-1,0,0,0,0)$, and we can write $S^{28}_5$ in the form $(C^+ + \mathbf{v}^+) \cap (C^- + \mathbf{v}^-)$ for the two cones $C^+ = \{ \mathbf{x} \in \R^5 \colon A^+ \mathbf{x} \le \mathbf{0} \}$ and $C^- = \{ \mathbf{x} \in \R^5 \colon A^- \mathbf{x} \le \mathbf{0} \}$.
Our arguments for bounding the circuit length of $S_5^{48}$ easily carry over to $S_5^{28}$ by analyzing the $2$-faces of $S_5^{28}$. The following result is the analogous statement to \cref{thm:santos-length,cor:santos-length-realization}. 

\begin{corollary} \label{cor:santos-23-length}
The circuit length of $S_5^{28}$ is at most $5$. The same bound holds for all realizations of $S_5^{28}$ with facets and facet-defining inequalities labeled $\mathsf{1}^\pm$ to $\mathsf{14}^\pm$ such that each facet is combinatorially equivalent to the facet of $S_5^{28}$ with the same label, 
and such that
\begin{myenumerate}
    \item the polyhedron given by inequalities $\mathsf{8^\pm}$ to $\mathsf{14^\pm}$ together with one of the pairs $\mathsf{2^\pm}$ or $\mathsf{3^\pm}$ is bounded, and \label{enum:santos-23-length-realization-i}
    \item the polyhedron given by inequalities $\mathsf{1^\pm}$ to $\mathsf{7^\pm}$ together with one of the pairs $\mathsf{9^\pm}$ or $\mathsf{10^\pm}$ is bounded. \label{enum:santos-23-length-realization-ii}
\end{myenumerate}
\end{corollary}
\begin{proof}
The proof strategy is identical to the proofs of \cref{thm:santos-length,cor:santos-length-realization}. We only give the necessary modifications here.

Each apex of $S_5^{28}$ is contained in a $2$-face with a vertex at distance $3$ from the other apex in the graph of $S_5^{28}$ (see the graphs in \cref{fig:santos-23-face} in the Appendix). The pairs of facet-defining inequalities $\mathsf{2^\pm}$, $\mathsf{3^\pm}$, $\mathsf{9^\pm}$, and $\mathsf{10^\pm}$ now take the role that $\mathsf{3^\pm}$, $\mathsf{4^\pm}$, $\mathsf{15^\pm}$, and $\mathsf{16^\pm}$ took for $S_5^{48}$ (cf.~ \cref{cor:santos-length-realization}). Hence, properties \cref{enum:santos-23-length-realization-i,enum:santos-23-length-realization-ii} are the analogues of Claim ($\ast$) from the proof of \cref{thm:santos-length} (and properties \cref{enum:santos-length-realization-i} and \cref{enum:santos-length-realization-ii} in \cref{cor:santos-length-realization}) and are therefore sufficient conditions for the existence of circuit walks of length at most $5$ between the apices of any realization of $S_5^{28}$.

It remains to show that $S^{28}_5$ itself satisfies properties \cref{enum:santos-23-length-realization-i,enum:santos-23-length-realization-ii}. Note that $S_5^{28}$ is not symmetric under the transformations in \cref{prop:santos-symmetry}: to switch the roles of the apices while leaving $S^{28}_5$ invariant, rows $\mathsf{4^\pm}$ and $\mathsf{5^\pm}$ would have to correspond with $\mathsf{11^\pm}$ and $\mathsf{12^\pm}$, which is impossible to achieve by permutating coordinates and flipping signs. Similarly, no permutation of coordinates (except for the identity) preserves $A^\pm$. 
This means that we cannot use the same argument as in the proof of \cref{cor:santos-length-realization} to reduce all four pairs of inequalities $\mathsf{2^\pm}$, $\mathsf{3^\pm}$, $\mathsf{9^\pm}$, and $\mathsf{10^\pm}$ to just one.
However, summing over all rows of $A^+$ yields $(14,0,0,0,0)$, which implies that this vector is in the strict interior of the polar cone of $C^+$. Similarly, the sum of all rows of $A^-$, which is the vector $(-14,0,0,0,0)$, is in the strict interior of the polar cone of $C^-$. 
For each of the pairs $\mathsf{2^\pm}$, $\mathsf{3^\pm}$, $\mathsf{9^\pm}$, and $\mathsf{10^\pm}$, the two corresponding rows of $A^+$ or $A^-$ add up to $(\pm 2,0,0,0,0)$, respectively. So by the same argument as in the proof of Claim ($\ast$) in the proof of \cref{thm:santos-length}, it then follows that, in fact, all four polyhedra described in \cref{enum:santos-23-length-realization-i,enum:santos-23-length-realization-ii} are bounded. \qed 
\end{proof}

As for $S_5^{48}$, we are able to establish a circuit length of exactly $2$ for the particular realization of $S_5^{28}$ in \cite[Corollary 2.9]{msw-15}.

\begin{theorem} \label{cor:santos-23-width}
The circuit length of $S^{28}_5$ is $2$.
\end{theorem}
\begin{proof}
A direct computation shows that the five inequalities $\mathsf{3^+}$, $\mathsf{6^+}$, $\mathsf{10^+}$, $\mathsf{11^+}$, and $\mathsf{13^+}$ define a vertex $\mathbf{y} = (0, \frac{1}{30}, \frac{2}{90}, \frac{2}{90}, \frac{1}{30})$ of $S_5^{28}$ (one of the highlighted vertices in \cref{fig:santos-23-face-a}). By a direct computation, one can verify that both difference vectors $90 (\mathbf{y} - \mathbf{v}^\pm) = (\pm 90, 3,2,2,3)$ are circuits: $(90, 3,2,2,3)$ is parallel to facets $\mathsf{3^-}$, $\mathsf{7^-}$,  $\mathsf{10^+}$, and $\mathsf{13^+}$; %\matthias{and 11+}
and $(-90, 3,2,2,3)$ is parallel to $\mathsf{3^+}$, $\mathsf{6^+}$, $\mathsf{10^-}$, and $\mathsf{14^-}$. %$\matthias{and 12-, both are redundant} 
Hence, $\mathbf{v}^+, \mathbf{y}, \mathbf{v}^-$ is a reversible circuit walk of length $2$.
Since the vector $(\mathbf{v}^+ - \mathbf{v}^-)/2 = (1,0,0,0,0)$ is not contained in any facet of $S_5^{28}$, %$C^+$ or $C^-$, 
it cannot be a circuit of $S_5^{28}$. It follows that the circuit length of $S_5^{28}$ is $2$. \qed
\end{proof}

Our framework for bounding the circuit lengths of $S^{48}_5$ and $S^{28}_5$ also applies to the remaining spindle from \cite[Theorem 2.14]{msw-15}, which has 25 facets and is given by $S_5^{25} = \{ \mathbf{x} \in \R^5 \colon A^+ \mathbf{x} \le \mathbf{1}, A^- \mathbf{x} \le \mathbf{1} \}$ where
\begingroup
\renewcommand*{\arraystretch}{1.1}
\[ \makeatletter\setlength\BA@colsep{4pt}\makeatother
  A^+ =\; \begin{blockarray}{cccccc}
  \begin{block}{(ccccc)l}
  1 & 0 & 0 & 0 & 32 & \;\mathsf{ 1} \\
  1 & 0 & 0 & 0 & -32 & \;\mathsf{ 2} \\
  1 & 0 & 0 & 21 & -7 & \;\mathsf{ 3} \\
  1 & 0 & 0 & -21 & -7 & \;\mathsf{ 4} \\
  1 & 0 & 0 & 20 & -4 & \;\mathsf{ 5} \\
  1 & 0 & 0 & -20 & -4 & \;\mathsf{ 6} \\
  1 & 0 & 0 & 16 & -15 & \;\mathsf{ 7} \\
  1 & 0 & 0 & -16 & -15 & \;\mathsf{ 8} \\
  1 & \frac{3}{50} & -\frac{1}{25} & 0 & -30 & \;\mathsf{ 9} \\
  1 & -\frac{3}{50} & -\frac{1}{25} & 0 & 30 & \;\mathsf{10} \\
  1 & \frac{3}{1000} & \frac{7}{1000} & 0 & -\frac{159}{5} & \;\mathsf{11} \\
  1 & -\frac{3}{1000} & \frac{7}{1000} & 0 & \frac{159}{5} & \;\mathsf{12} \\
  \end{block}
  \end{blockarray},\;
  %
  %\text{ and }
  %
  A^- =\; \begin{blockarray}{cccccc}
  \begin{block}{(ccccc)l}
  -1 & 60 & 0 & 0 & 0 & \;\mathsf{13} \\
  -1 & -55 & 0 & 0 & 0 & \;\mathsf{14} \\
  -1 & 0 & 76 & 0 & 0 & \;\mathsf{15} \\
  -1 & 0 & -33 & 0 & 0 & \;\mathsf{16} \\
  -1 & 44 & 34 & 0 & 0 & \;\mathsf{17} \\
  -1 & 8 & -30 & 0 & 0 & \;\mathsf{18} \\
  -1 & -34 & 36 & 0 & 0 & \;\mathsf{19} \\
  -1 & -2 & -32 & 0 & 0 & \;\mathsf{20} \\
  -1 & -20 & 0 & \frac{1}{5} & -\frac{1}{5} & \;\mathsf{21} \\
  -1 & \frac{2999}{50} & 0 & -\frac{3}{25} & -\frac{1}{5} & \;\mathsf{22} \\
  -1 & \frac{299999}{5000} & 0 & 0 & \frac{1}{100} & \;\mathsf{23} \\
  -1 & -\frac{549}{10} & 0 & \frac{1}{5000} & \frac{1}{800} & \;\mathsf{24} \\
  -1 & -54 & 0 & \frac{1}{500} & -\frac{1}{80} & \;\mathsf{25} \\
  \end{block}
  \end{blockarray}.
\]
\endgroup

The two apices of $S^{25}_5$ again are $\mathbf{v}^+ = (1,0,0,0,0)$ and $\mathbf{v}^- = (-1,0,0,0,0)$.
For the two cones $C^+ = \{ \mathbf{x} \in \R^5 \colon A^+ \mathbf{x} \le \mathbf{0} \}$ and $C^- = \{ \mathbf{x} \in \R^5 \colon A^- \mathbf{x} \le \mathbf{0} \}$, we can write $S^{25}_5$ as $S^{25}_5 = (C^+ + \mathbf{v}^+) \cap (C^- + \mathbf{v}^-)$.

\begin{corollary} \label{cor:weibel-length} 
The circuit length of $S_5^{25}$ is at most $5$. The same bound holds for all realizations of $S_5^{25}$ with facets and facet-defining inequalities labeled $\mathsf{1}$ to $\mathsf{25}$ such that each facet is combinatorially equivalent to the facet of $S_5^{25}$ with the same label, and such that
\begin{myenumerate}
    \item the polyhedron given by inequalities $\mathsf{13}$ to $\mathsf{25}$ together with one of the pairs $\mathsf{1}, \mathsf{2}$ or $\mathsf{3}, \mathsf{4}$ is bounded, and \label{enum:weibel-length-realization-i}
    \item the polyhedron given by inequalities $\mathsf{1}$ to $\mathsf{12}$ together with one of the pairs $\mathsf{13}, \mathsf{14}$ or $\mathsf{15}, \mathsf{16}$ is bounded. \label{enum:weibel-length-realization-ii}
\end{myenumerate}
\end{corollary}
\begin{proof}
Examples of the relevant $2$-faces and their graphs are given in \cref{fig:weibel-face} in the Appendix. Note that for each such face, the two critical facet-defining inequalities are one of the pairs $\mathsf{1}, \mathsf{2}$, or $\mathsf{3}, \mathsf{4}$, or $\mathsf{13}, \mathsf{14}$, or $\mathsf{15}, \mathsf{16}$ in \cref{enum:weibel-length-realization-i,enum:weibel-length-realization-ii}. The proof of the second part of the statement is therefore analogous to the proofs for $S_5^{48}$ and $S_5^{28}$ above (\cref{thm:santos-length,cor:santos-length-realization,cor:santos-23-length}).

It remains to prove that $S_5^{25}$ satisfies \cref{enum:weibel-length-realization-i,enum:weibel-length-realization-ii}. We again follow the proof strategy for Claim ($\ast$) in the proof of \cref{thm:santos-length}. We show that for each of the pairs $\mathsf{1}, \mathsf{2}$ and $\mathsf{13}, \mathsf{14}$, there is a non-negative linear combination of the two corresponding rows of $A^+$ or $A^-$ whose negative is in the strict interior of the polar cone of $C^-$ or $C^+$, respectively. It then follows that the respective polyhedra in \cref{enum:weibel-length-realization-i,enum:weibel-length-realization-ii} are bounded, % in the case of $S_5^{25}$ 
since their recession cones only consist of $\mathbf{0}$.

To see this, first observe that we can write the vector $(-1,0,0,0,0)$ as a non-negative linear combination of rows $\mathsf{13}$ and $\mathsf{14}$ of $A^-$. (In fact, the same holds true for rows $\mathsf{15}$ and $\mathsf{16}$.) Further, consider a linear combination of the rows of $A^+$ that assigns coefficient $\frac{21}{8}$ for row $\mathsf{1}$, $\frac{40}{7}$ for rows $\mathsf{11}$ and $\mathsf{12}$, and coefficient $1$ for all other rows. 
This linear combination yields a positive multiple of $(1,0,0,0,0)$. Since all coefficients are positive, we conclude that $(1,0,0,0,0)$ is in the strict interior of $C^+$.

For rows $\mathsf{1}$ and $\mathsf{2}$ of $A^+$, we proceed analogously. Their sum equals $(2,0,0,0,0)$. To show that $(-2,0,0,0,0)$ is in the strict interior of the polar cone of $C^-$, consider the linear combination of rows $\mathsf{22}$, $\mathsf{24}$, and $\mathsf{25}$ of $A^-$ with coefficients $1$, $380$, and $22$, respectively. The resulting vector is $(-403,-\beta,0,0,0)$ where $\beta = 549 \cdot 38 + 54 \cdot 22 - \frac{2999}{50} > 0$. Further, the linear combination of rows $\mathsf{16}$ and $\mathsf{17}$ of $A^-$ with coefficients $\frac{34}{33 \cdot 44} \beta$ and $\frac{1}{44} \beta$, respectively, yields the vector $(-\frac{67}{33 \cdot 44} \beta, \beta, 0,0,0)$. Adding both vectors, we thus obtain a positive multiple of $(-1,0,0,0,0)$ from a linear combination of rows $\mathsf{16}$, $\mathsf{17}$, $\mathsf{22}$, $\mathsf{24}$, and $\mathsf{25}$ of $A^-$ where all coefficients are positive since $\beta>0$. This means that the vector $(-1,0,0,0,0)$ is in the strict interior of the cone generated by rows $\mathsf{16}$, $\mathsf{17}$, $\mathsf{22}$, $\mathsf{24}$, and $\mathsf{25}$ of $A^-$. Note that this cone is contained in the polar cone of $C^-$, which is generated by all rows of $A^-$, and it is full-dimensional since the five rows $\mathsf{16}$, $\mathsf{17}$, $\mathsf{22}$, $\mathsf{24}$, and $\mathsf{25}$ are linearly independent. Hence, $(-1,0,0,0,0)$ must also be in the strict interior of the polar cone of $C^-$. \qed
\end{proof}

In contrast to the statements of \cref{cor:santos-width,cor:santos-23-width} for the other two spindles, the circuit length of $S_5^{25}$ as given in \cite[Theorem 2.14]{msw-15} is at least $3$. This can be verified computationally by a brute-force enumeration of all points $\mathbf{y}_1$ on the boundary of $S_5^{25}$ that can be reached from $\mathbf{v}^-$ via a single circuit step ($S_5^{25}$ has 17454 circuits). For no such point $\mathbf{y}_1$, the vector $\mathbf{v}^+-\mathbf{y}_1$ is a circuit direction. 

We conclude this section with some remarks on our proofs. For all three spindles $S_5^{48}$, $S_5^{28}$, and $S_5^{25}$, the faces given in \cref{fig:santos-face,fig:santos-23-face,fig:weibel-face} are not the only $2$-faces that satisfy the prerequisite of \cref{lem:2-face}. In fact, we enumerated all $2$-faces using Polymake \cite{polymake} and found that for each of the three spindles there are $32$ such $2$-faces that contain one of the apices ($16$ for each apex). Each of them is combinatorially equivalent to one of the examples given in \cref{fig:santos-face,fig:santos-23-face,fig:weibel-face}. Moreover, for any such $2$-face, the two facet-defining inequalities that are relevant for verifying the boundedness condition in \cref{lem:2-face} are one of the pairs given in \cref{cor:santos-length-realization,cor:santos-23-length,cor:weibel-length}.

Finally, as explained above, our bounds on the circuit length of $S_5^{48}$, $S_5^{28}$, and $S_5^{25}$ are robust under mild perturbations as long as they retain the properties in \cref{cor:santos-length-realization,cor:santos-23-length,cor:weibel-length}, respectively. However, we do not know whether, in fact, \emph{all} realizations of the three spindles satisfy these properties. We leave this as an open question.

\subsection{An Analytical and Computational Verification for the $20$- and $23$-Dimensional Hirsch Counterexamples} \label{sec:Santos-highdim}

Santos' original disproof of the bounded Hirsch conjecture in \cite{s-11} crucially relies on finding a degenerate spindle whose (combinatorial) length is greater than its dimension. We have shown that in the circuit setting, neither Santos' original spindle $S_5^{48}$ nor any of the subsequent improvements $S_5^{28}$ and $S_5^{25}$ from \cite{msw-15} meet this requirement: all three spindles (and slight perturbations thereof) have circuit length at most $5$. Therefore, neither of them leads to a counterexample to the circuit diameter conjecture via Santos' construction from \cite{s-11}. %\matthias{Is this too dangerous to say? (Steffen) I changed `can lead' to `leads'. in my opinion this is safe to say}
For the two smaller of the three spindles, the steps of this construction have been explicitly carried out by Matschke, Santos, and Weibel \cite{msw-15}, resulting in inequality descriptions of two explicit Hirsch counterexamples. Our arguments developed in \cref{sec:Santos-5} even allow us to verify that the circuit length of these two explicitly given spindles is indeed at most their dimension. To see how our techniques also apply here, we first explain Santos' construction in more detail. As the original construction in \cite{s-11} is stated in terms of prismatoids, the polar duals of spindles, we briefly repeat it in the language of spindles here. 

Let $S_d^f \subset \R^d$ be a $d$-dimensional spindle with $f$ facets and length $l$ where $f>2d$ and $l>d$. We denote the apices of $S_d^f$ by $\mathbf{u}$ and $\mathbf{v}$. Since $f>2d$, at least one of the apices $\mathbf{u}$ is degenerate. Now choose an arbitrary facet $F$ of $S_d^f$ that contains the other apex $\mathbf{v}$ and perform the following \emph{wedge operation}:
Let $H^+$ and $H^-$ be two (non-parallel) hyperplanes in $\R^{d+1}$ such that each of them intersects the interior of $S_d^f \times \R \subset \R^{d+1}$ and $H^+ \cap H^- \supseteq F \times \{0\}$. Then for the two polyhedra $W^\pm$ given by $W^\pm = (S_d^f \times \R) \cap H^\pm$, we define $W_F(S_d^f) = \conv(W^+ \cup W^-)$. Note that by construction, $W^\pm$ are affinely equivalent embeddings of the spindle $S_d^f$ into the two hyperplanes $H^\pm$. Hence, $W_F(S_d^f)$ is a polytope again. We call $W_F(S_d^f)$ a \emph{wedge} (on $S_d^f$) \emph{over the facet $F$}. See \cref{fig:wedge} for an illustration of the wedge operation. Note that this operation can increase the circuit diameter by at most one \cite{bsy-18}. 

The wedge $W_F(S_d^f)$ has $f+1$ facets in dimension $d+1$ and is almost a spindle: each facet either contains the vertex $(\mathbf{v},0)$ or the edge between $\mathbf{u}^+$ and $\mathbf{u}^-$, where $\mathbf{u}^\pm$ denotes the apex of $W^\pm$ distinct from $(\mathbf{v},0)$. To get a spindle from $W_F(S_d^f)$, we carefully perturb the facets of $W_F(S_d^f)$ that contain the edge between $\mathbf{u}^\pm$ so as to make an interior point of this edge become a vertex (the new apex; see \cref{fig:wedge}c). If the perturbation is done appropriately as described in \cite{s-11}, the resulting spindle $S_{d+1}^{f+1}$ has length at least $l+1$. In fact, by the proof of Theorem 2.6 in \cite{s-11}, carefully perturbing a \emph{single} facet suffices to increase the length as desired.

If this wedge-plus-perturbation operation is iteratively applied $f-2d$ times to $S_d^f$, we obtain an $(f-d)$-dimensional spindle $S_{f-d}^{2f-2d}$ with $2f-2d$ facets and length at least $l+f-2d$. So if $l>d$, then the length of $S_{f-d}^{2f-2d}$ exceeds $f-d$, which means that the spindle $S_{f-d}^{2f-2d}$ violates the Hirsch conjecture.

\begin{figure}[htb]
\includegraphics[]{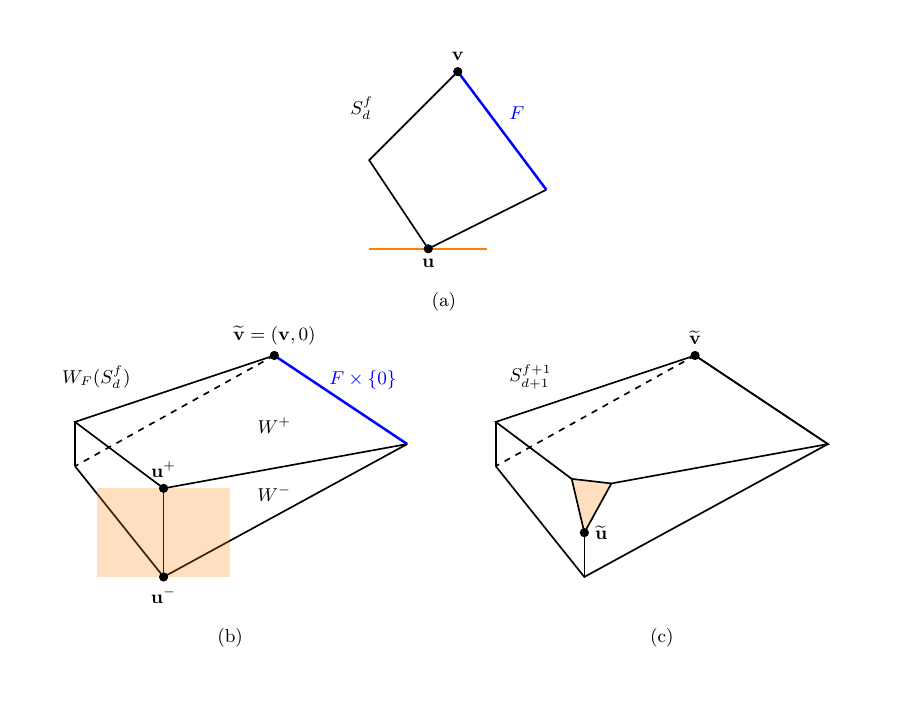}
\vspace*{-1.5cm}
\caption{(a) The initial spindle $S_d^f$ is degenerate (the orange line at $\mathbf{u}$ indicates an `extra' facet incident with $\mathbf{u}$). (b) By wedging over a facet $F$ that contains the other apex $\mathbf{v}$, we obtain the wedge $W_F(S_d^f)$ with two facets $W^\pm$ that are affinely equivalent with $S_d^f$. The adjacent vertices $\mathbf{u}^\pm$ now correspond to the apex $\mathbf{u}$ of $S_d^f$. (c) By perturbing the orange facet of the wedge, we get a spindle $S_{d+1}^{f+1}$ with apices $\widetilde{\mathbf{u}}$ and $\widetilde{\mathbf{v}}$.}\vspace*{-0.4cm}
\label{fig:wedge}
\end{figure}

In \cite{msw-15}, Matschke, Santos, and Weibel explicitly built and computationally checked two Hirsch counterexamples resulting from $S_5^{28}$ and $S^{25}_5$ via Santos' construction described above. The resulting spindles are of length $24$ and $21$ in dimension $23$ and $20$, respectively. The authors remark that carrying out the steps of the construction in such a way that the length indeed increases as desired was computationally feasible only for the two smaller spindles $S_5^{28}$ and $S^{25}_5$ and not for $S_5^{48}$ (see also Santos' remark in \cite[Section 1]{s-11}). For those two spindles, we verified computationally that our proof technique from \cref{sec:Santos-5} for bounding their circuit length also transfers to the explicit counterexamples themselves obtained by Matschke, Santos, and Weibel. %\steffen{fantastic. you framed this very well} 

Using the inequality descriptions and vertex adjacencies provided in \cite{msw-15,w-url}, we found that slight perturbations of the $2$-faces in \cref{fig:santos-23-face,fig:weibel-face} still appear as $2$-faces (with the same combinatorics) after the final wedge-plus-perturbation step. Furthermore, our computations show that on the final spindle, the length of a shortest edge walk from each apex to those $2$-faces increases by exactly the number of times we wedge over a facet that contains the apex.
For instance, the $20$-dimensional explicit counterexample from \cite[Theorem 1.3]{msw-15} based on $S_5^{25}$ has a $2$-face that is a perturbed equivalent of the $2$-face in \cref{fig:weibel-face-c}. The original face of $S_5^{25}$ could be reached within $3$ edge steps from one apex, and the other apex was a vertex of the face already. Now, in dimension $20$, the equivalent face can be reached within $3+8=11$ edge steps from one apex, and its vertex that was the other apex in dimension $5$ now is at distance $7$ from the new apex. 
The perturbations applied by Matschke, Santos, and Weibel are small enough for the properties \cref{enum:weibel-length-realization-i,enum:weibel-length-realization-ii} in \cref{cor:weibel-length} to still hold. By our arguments in \cref{sec:Santos-5}, this means that there is a circuit walk of length at most $2$ on the perturbed $2$-face that connects the two edge walks to and from the face to give a circuit walk of total length at most $11+2+7=20$. Also for the other $2$-faces of $S_5^{25}$ in \cref{fig:weibel-face} and those of $S_5^{28}$ in \cref{fig:santos-23-face}, we verified computationally that the length $5$ circuit walks via those faces can be extended in a completely analogous way to obtain circuit walks of the desired length in higher dimension. As an immediate consequence of these observations, we obtain the following corollaries.

\begin{corollary}
The circuit length of the $20$-dimensional spindle with $40$ facets from \cite[Theorem 1.3]{msw-15} is at most $20$.
\end{corollary}

\begin{corollary}
The circuit length of the $23$-dimensional spindle with $46$ facets from \cite{msw-15,w-url} is at most $23$.
\end{corollary}

We stress that these two explicit Hirsch counterexamples result from a particular sequence of wedge-plus-perturbation operations applied to $S_5^{25}$ and $S_5^{28}$, respectively. However, the steps of Santos' construction are not uniquely determined: the choice of the facet to wedge over is arbitrary (as long as it contains the right apex), and so is the choice of the facet that is perturbed. Different choices may lead to different counterexamples. Nonetheless, our arguments from \cref{sec:Santos-5} enable us to make the following observation: regardless of how the steps of Santos' construction applied to $S^{28}_5$ or $S_5^{28}$ are executed, the $2$-faces that our circuit length bounds for the $5$-dimensional spindles crucially relied on will be preserved up to slight changes.

To see this, consider the first wedge on $S_5^{25}$ (or $S_5^{28}$) over an arbitrary facet (both apices are degenerate). Let us denote the two facets that are affinely equivalent with $S_5^{25}$ by $W^\pm$ with apices $\mathbf{v}$ and $\mathbf{u}^\pm$, as in the sketch in \cref{fig:wedge}b. Thus, all $2$-faces of $S_5^{25}$ in \cref{fig:weibel-face} also appear as $2$-faces of $W^\pm$ (up to an affine transformation). If we now perturb a facet according to Santos' construction (one of the facets that contains $\mathbf{u}^\pm$), then one of the vertices $\mathbf{u}^\pm$, say $\mathbf{u}^+$, must be cut off in order to get a spindle again (cf.~\cref{fig:wedge}c). Note that the only degenerate vertices of $S_5^{25}$, and therefore of $W^+$, are the apices (this can be verified computationally, e.g., using Polymake \cite{polymake}). So by a slight perturbation of the chosen facet, a $2$-face of $W^+$ that contains $\mathbf{u}^+$ will either be slightly perturbed without changing the combinatorics, or it will become a $2$-face where combinatorially the only change is that $\mathbf{u}^+$ is replaced with two new, adjacent vertices (the edge between them must then be defined by the perturbed facet). Moreover, $2$-faces of $W^+$ that do not contain $\mathbf{u}^+$ are unaffected (up to slight perturbations) since the facet that we perturb contains $\mathbf{u}^+$.
In either case, \cref{lem:2-face} guarantees that on the resulting $2$-face, two circuit steps still suffice to reach a vertex that corresponds to the apex of $S_5^{25}$ contained in the corresponding face of $S_5^{25}$. 
The above observation also applies to $S_5^{28}$ by noting that all vertices of $S_5^{28}$ other than the apices are non-degenerate. However, this is not true for the vertices of $S_5^{48}$. Therefore, we cannot directly conclude that any wedge-plus-perturbation operation according to Santos' construction will preserve the $2$-face in \cref{fig:santos-face} or its symmetric equivalents in the proofs of \cref{thm:santos-length,cor:santos-length-realization}.

%---------------------------------------------------------------------

\subsection*{Acknowledgements}
 We would like to thank Nicholas Crawford, Jes\'{u}s A. De Loera, Michael Wigal, and Youngho Yoo for insightful discussions.

%---------------------------------------------------------------------
\bibliographystyle{splncs04}

\clearpage
\section*{Appendix} In our discussion in \cref{sec:Todd}, we identified four additional orientations for the Todd polytope $M_4$ that can be used to disprove the monotone Hirsch conjecture if combined with an objective function uniquely minimized at $\mathbf{0}$. We display these orientations in Figures \ref{fig:Todd2}, \ref{fig:Todd3}, \ref{fig:Todd4}, and \ref{fig:Todd5}.

In \cref{sec:Santos-5}, one of the key arguments for bounding the circuit length of the $5$-dimensional spindles was to exhibit the existence of a $2$-face $F$ incident to the target apex that can be reached in three edge steps from the other. \cref{fig:santos-face} displayed the corresponding subgraph for $S_5^{48}$. In \cref{fig:santos-23-face,fig:weibel-face}, we display the corresponding subgraphs for $S_5^{28}$ and $S_5^{25}$.

%%%%%%%%%
\def\yscale{.8cm}
%%%%%%%%%
\vfill

\begin{figure}
    \centering
    \begin{tikzpicture}
        \draw[thick] (0, 0) node[orange, circle,fill, inner sep = 1.5pt] {};
        \draw[thick] (-1, 0) node[orange, circle,fill, inner sep = 1.5pt] {};
        \draw[thick] (-2, 1) node[orange, circle,fill, inner sep = 1.5pt] {};
        \draw[thick] (1, 0) node[orange, circle,fill, inner sep = 1.5pt] {};
        \draw[thick] (2, 1) node[orange, circle,fill, inner sep = 1.5pt] {};
        \draw[thick] (2, 3) node[orange, circle,fill, inner sep = 1.5pt] {};
        \draw[thick] (3, 3) node[orange, circle,fill, inner sep = 1.5pt] {};
        \draw[thick] (-3, 3) node[orange, circle,fill, inner sep = 1.5pt] {};
        \draw[thick] (-2, 3) node[orange, circle,fill, inner sep = 1.5pt] {};
        \draw[thick] (-4, 9) node[orange, circle,fill, inner sep = 1.5pt] {};
        \draw[thick] (4, 9) node[orange, circle,fill, inner sep = 1.5pt] {};
        \draw[thick] (0, 3) node[orange, circle,fill, inner sep = 1.5pt] {};
        \draw[thick] (-2, 5) node[orange, circle,fill, inner sep = 1.5pt] {};
        \draw[thick] (-3, 5) node[orange, circle,fill, inner sep = 1.5pt] {};
        \draw[thick] (3, 5) node[orange, circle,fill, inner sep = 1.5pt] {};
        \draw[thick] (2, 5) node[orange, circle,fill, inner sep = 1.5pt] {};
        \draw[thick] (-2, 7) node[orange, circle,fill, inner sep = 1.5pt] {};
        \draw[thick] (2, 7) node[orange, circle,fill, inner sep = 1.5pt] {};
        \draw[thick] (0, 8) node[orange, circle,fill, inner sep = 1.5pt] {};
        \draw[thick] (0, 10) node[orange, circle,fill, inner sep = 1.5pt] {};
        \draw[thick, ->] (-1/10, 1/10) -- (-9/5, 4/5);
        \draw[thick, ->] (1/10, 1/10) -- (9/5, 4/5);
        \draw[thick, ->] (21/10, 5) -- (14/5, 5);
        \draw[thick, ->] (2, 51/10) -- (2, 34/5);
        \draw[thick, ->] (-3, 31/10) -- (-3, 24/5);
        \draw[thick, ->] (-29/10, 3) -- (-11/5, 3);
        \draw[thick, ->] (-9/10, 0) -- (-1/5, 0);
        \draw[thick, ->] (-9/10, 1/10) -- (9/5, 24/5);
        \draw[thick, ->] (-11/10, 1/10) -- (-9/5, 4/5);
        \draw[thick, ->] (-29/10, 49/10) -- (-1/5, 16/5);
        \draw[thick, ->] (-29/10, 51/10) -- (-11/5, 34/5);
        \draw[thick, ->] (29/10, 51/10) -- (11/5, 34/5);
        \draw[thick, ->] (1/10, 31/10) -- (14/5, 24/5);
        \draw[thick, ->] (-21/10, 11/10) -- (-14/5, 14/5);
        \draw[thick, ->] (-2, 11/10) -- (-2, 14/5);
        \draw[thick, ->] (19/10, 71/10) -- (1/5, 39/5);
        \draw[thick, ->] (-19/10, 31/10) -- (9/5, 24/5);
        \draw[thick, ->] (-19/10, 3) -- (-1/5, 3);
        \draw[thick, ->] (39/10, 89/10) -- (11/5, 36/5);
        \draw[thick, ->] (39/10, 89/10) -- (1/5, 41/5);
        \draw[thick, ->] (39/10, 89/10) -- (16/5, 16/5);
        \draw[thick, ->] (19/10, 3) -- (1/5, 3);
        \draw[thick, ->] (19/10, 31/10) -- (-9/5, 24/5);
        \draw[thick, ->] (3, 31/10) -- (3, 24/5);
        \draw[thick, ->] (29/10, 3) -- (11/5, 3);
        \draw[thick, ->] (-21/10, 5) -- (-14/5, 5);
        \draw[thick, ->] (-2, 51/10) -- (-2, 34/5);
        \draw[thick, ->] (-19/10, 71/10) -- (-1/5, 39/5);
        \draw[thick, ->] (-1/10, 99/10) -- (-4/5, 1/5);
        \draw[thick, ->] (1/10, 99/10) -- (19/5, 46/5);
        \draw[thick, ->] (-1/10, 99/10) -- (-19/5, 46/5);
        \draw[thick, ->] (1/10, 99/10) -- (4/5, 1/5);
        \draw[thick, ->] (-39/10, 89/10) -- (-16/5, 16/5);
        \draw[thick, ->] (-39/10, 89/10) -- (-1/5, 41/5);
        \draw[thick, ->] (-39/10, 89/10) -- (-11/5, 36/5);
        \draw[thick, ->] (9/10, 0) -- (1/5, 0);
        \draw[thick, ->] (9/10, 1/10) -- (-9/5, 24/5);
        \draw[thick, ->] (11/10, 1/10) -- (9/5, 4/5);
        \draw[thick, ->] (2, 11/10) -- (2, 14/5);
        \draw[thick, ->] (21/10, 11/10) -- (14/5, 14/5);
    \end{tikzpicture}
    \caption{The digraph of $M_{4}$ for the orientation $\mathbf{c} = (10716,13680,3477,4465)$}
    \label{fig:Todd2}
\end{figure}
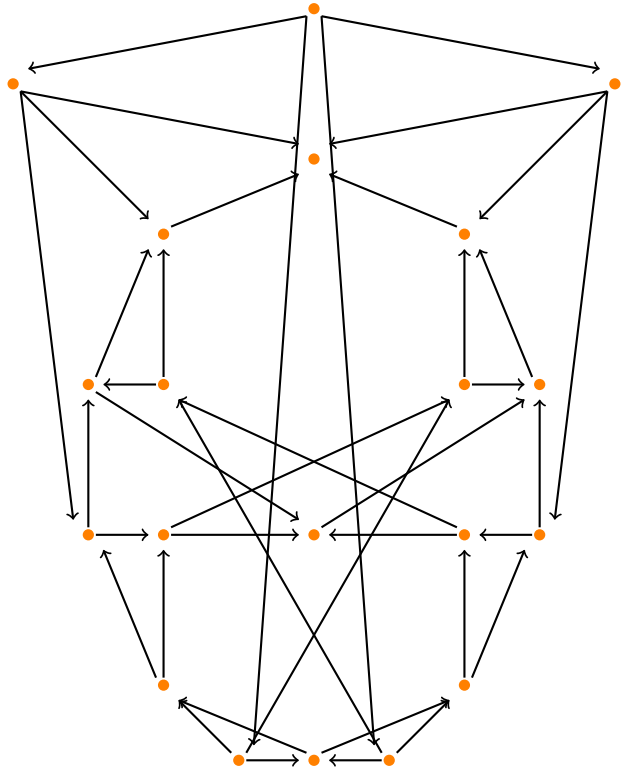

\vfill

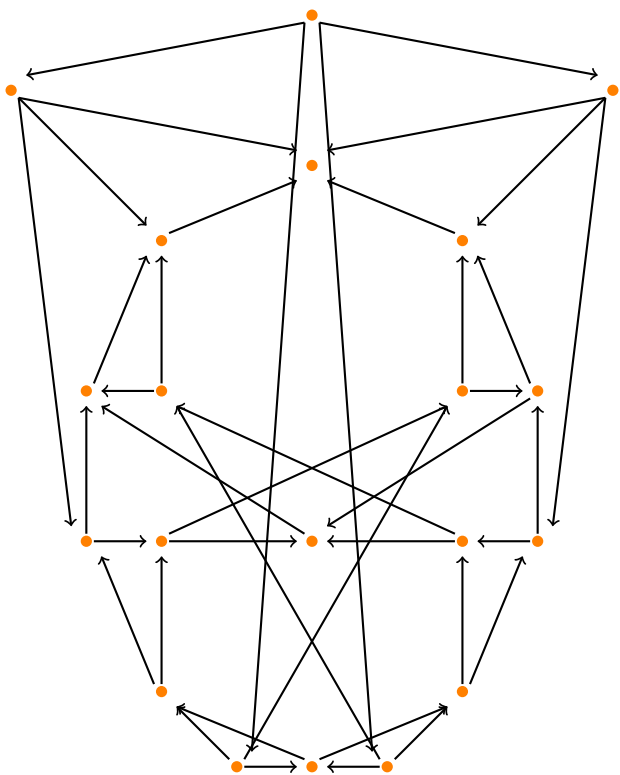
\begin{figure}
    \centering
    \begin{tikzpicture}
        \draw[thick] (0, 0) node[orange, circle,fill, inner sep = 1.5pt] {};
        \draw[thick] (-1, 0) node[orange, circle,fill, inner sep = 1.5pt] {};
        \draw[thick] (-2, 1) node[orange, circle,fill, inner sep = 1.5pt] {};
        \draw[thick] (1, 0) node[orange, circle,fill, inner sep = 1.5pt] {};
        \draw[thick] (2, 1) node[orange, circle,fill, inner sep = 1.5pt] {};
        \draw[thick] (2, 3) node[orange, circle,fill, inner sep = 1.5pt] {};
        \draw[thick] (3, 3) node[orange, circle,fill, inner sep = 1.5pt] {};
        \draw[thick] (-3, 3) node[orange, circle,fill, inner sep = 1.5pt] {};
        \draw[thick] (-2, 3) node[orange, circle,fill, inner sep = 1.5pt] {};
        \draw[thick] (-4, 9) node[orange, circle,fill, inner sep = 1.5pt] {};
        \draw[thick] (4, 9) node[orange, circle,fill, inner sep = 1.5pt] {};
        \draw[thick] (0, 3) node[orange, circle,fill, inner sep = 1.5pt] {};
        \draw[thick] (-2, 5) node[orange, circle,fill, inner sep = 1.5pt] {};
        \draw[thick] (-3, 5) node[orange, circle,fill, inner sep = 1.5pt] {};
        \draw[thick] (3, 5) node[orange, circle,fill, inner sep = 1.5pt] {};
        \draw[thick] (2, 5) node[orange, circle,fill, inner sep = 1.5pt] {};
        \draw[thick] (-2, 7) node[orange, circle,fill, inner sep = 1.5pt] {};
        \draw[thick] (2, 7) node[orange, circle,fill, inner sep = 1.5pt] {};
        \draw[thick] (0, 8) node[orange, circle,fill, inner sep = 1.5pt] {};
        \draw[thick] (0, 10) node[orange, circle,fill, inner sep = 1.5pt] {};
        \draw[thick, ->] (-1/10, 1/10) -- (-9/5, 4/5);
        \draw[thick, ->] (1/10, 1/10) -- (9/5, 4/5);
        \draw[thick, ->] (21/10, 5) -- (14/5, 5);
        \draw[thick, ->] (2, 51/10) -- (2, 34/5);
        \draw[thick, ->] (-3, 31/10) -- (-3, 24/5);
        \draw[thick, ->] (-29/10, 3) -- (-11/5, 3);
        \draw[thick, ->] (-9/10, 0) -- (-1/5, 0);
        \draw[thick, ->] (-9/10, 1/10) -- (9/5, 24/5);
        \draw[thick, ->] (-11/10, 1/10) -- (-9/5, 4/5);
        \draw[thick, ->] (-29/10, 51/10) -- (-11/5, 34/5);
        \draw[thick, ->] (29/10, 49/10) -- (1/5, 16/5);
        \draw[thick, ->] (29/10, 51/10) -- (11/5, 34/5);
        \draw[thick, ->] (-1/10, 31/10) -- (-14/5, 24/5);
        \draw[thick, ->] (-21/10, 11/10) -- (-14/5, 14/5);
        \draw[thick, ->] (-2, 11/10) -- (-2, 14/5);
        \draw[thick, ->] (19/10, 71/10) -- (1/5, 39/5);
        \draw[thick, ->] (-19/10, 31/10) -- (9/5, 24/5);
        \draw[thick, ->] (-19/10, 3) -- (-1/5, 3);
        \draw[thick, ->] (39/10, 89/10) -- (11/5, 36/5);
        \draw[thick, ->] (39/10, 89/10) -- (1/5, 41/5);
        \draw[thick, ->] (39/10, 89/10) -- (16/5, 16/5);
        \draw[thick, ->] (19/10, 3) -- (1/5, 3);
        \draw[thick, ->] (19/10, 31/10) -- (-9/5, 24/5);
        \draw[thick, ->] (3, 31/10) -- (3, 24/5);
        \draw[thick, ->] (29/10, 3) -- (11/5, 3);
        \draw[thick, ->] (-21/10, 5) -- (-14/5, 5);
        \draw[thick, ->] (-2, 51/10) -- (-2, 34/5);
        \draw[thick, ->] (-19/10, 71/10) -- (-1/5, 39/5);
        \draw[thick, ->] (-1/10, 99/10) -- (-4/5, 1/5);
        \draw[thick, ->] (1/10, 99/10) -- (19/5, 46/5);
        \draw[thick, ->] (-1/10, 99/10) -- (-19/5, 46/5);
        \draw[thick, ->] (1/10, 99/10) -- (4/5, 1/5);
        \draw[thick, ->] (-39/10, 89/10) -- (-16/5, 16/5);
        \draw[thick, ->] (-39/10, 89/10) -- (-1/5, 41/5);
        \draw[thick, ->] (-39/10, 89/10) -- (-11/5, 36/5);
        \draw[thick, ->] (9/10, 0) -- (1/5, 0);
        \draw[thick, ->] (9/10, 1/10) -- (-9/5, 24/5);
        \draw[thick, ->] (11/10, 1/10) -- (9/5, 4/5);
        \draw[thick, ->] (2, 11/10) -- (2, 14/5);
        \draw[thick, ->] (21/10, 11/10) -- (14/5, 14/5);
    \end{tikzpicture}
    \caption{The digraph of $M_{4}$ for the orientation $\mathbf{c} = (13680, 10716, 4465, 3477)$.}
    \label{fig:Todd3}
\end{figure}

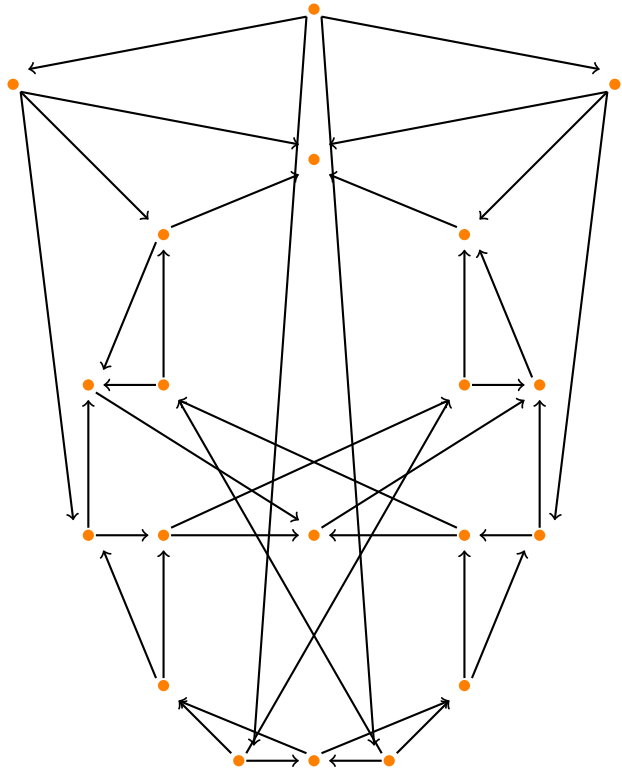
\begin{figure}
    \centering
    \begin{tikzpicture}
        \draw[thick] (0, 0) node[orange, circle,fill, inner sep = 1.5pt] {};
        \draw[thick] (-1, 0) node[orange, circle,fill, inner sep = 1.5pt] {};
        \draw[thick] (-2, 1) node[orange, circle,fill, inner sep = 1.5pt] {};
        \draw[thick] (1, 0) node[orange, circle,fill, inner sep = 1.5pt] {};
        \draw[thick] (2, 1) node[orange, circle,fill, inner sep = 1.5pt] {};
        \draw[thick] (2, 3) node[orange, circle,fill, inner sep = 1.5pt] {};
        \draw[thick] (3, 3) node[orange, circle,fill, inner sep = 1.5pt] {};
        \draw[thick] (-3, 3) node[orange, circle,fill, inner sep = 1.5pt] {};
        \draw[thick] (-2, 3) node[orange, circle,fill, inner sep = 1.5pt] {};
        \draw[thick] (-4, 9) node[orange, circle,fill, inner sep = 1.5pt] {};
        \draw[thick] (4, 9) node[orange, circle,fill, inner sep = 1.5pt] {};
        \draw[thick] (0, 3) node[orange, circle,fill, inner sep = 1.5pt] {};
        \draw[thick] (-2, 5) node[orange, circle,fill, inner sep = 1.5pt] {};
        \draw[thick] (-3, 5) node[orange, circle,fill, inner sep = 1.5pt] {};
        \draw[thick] (3, 5) node[orange, circle,fill, inner sep = 1.5pt] {};
        \draw[thick] (2, 5) node[orange, circle,fill, inner sep = 1.5pt] {};
        \draw[thick] (-2, 7) node[orange, circle,fill, inner sep = 1.5pt] {};
        \draw[thick] (2, 7) node[orange, circle,fill, inner sep = 1.5pt] {};
        \draw[thick] (0, 8) node[orange, circle,fill, inner sep = 1.5pt] {};
        \draw[thick] (0, 10) node[orange, circle,fill, inner sep = 1.5pt] {};
        \draw[thick, ->] (-1/10, 1/10) -- (-9/5, 4/5);
        \draw[thick, ->] (1/10, 1/10) -- (9/5, 4/5);
        \draw[thick, ->] (21/10, 5) -- (14/5, 5);
        \draw[thick, ->] (2, 51/10) -- (2, 34/5);
        \draw[thick, ->] (-3, 31/10) -- (-3, 24/5);
        \draw[thick, ->] (-29/10, 3) -- (-11/5, 3);
        \draw[thick, ->] (-9/10, 0) -- (-1/5, 0);
        \draw[thick, ->] (-9/10, 1/10) -- (9/5, 24/5);
        \draw[thick, ->] (-11/10, 1/10) -- (-9/5, 4/5);
        \draw[thick, ->] (-29/10, 49/10) -- (-1/5, 16/5);
        \draw[thick, ->] (29/10, 51/10) -- (11/5, 34/5);
        \draw[thick, ->] (1/10, 31/10) -- (14/5, 24/5);
        \draw[thick, ->] (-21/10, 11/10) -- (-14/5, 14/5);
        \draw[thick, ->] (-2, 11/10) -- (-2, 14/5);
        \draw[thick, ->] (19/10, 71/10) -- (1/5, 39/5);
        \draw[thick, ->] (-19/10, 31/10) -- (9/5, 24/5);
        \draw[thick, ->] (-19/10, 3) -- (-1/5, 3);
        \draw[thick, ->] (39/10, 89/10) -- (11/5, 36/5);
        \draw[thick, ->] (39/10, 89/10) -- (1/5, 41/5);
        \draw[thick, ->] (39/10, 89/10) -- (16/5, 16/5);
        \draw[thick, ->] (19/10, 3) -- (1/5, 3);
        \draw[thick, ->] (19/10, 31/10) -- (-9/5, 24/5);
        \draw[thick, ->] (3, 31/10) -- (3, 24/5);
        \draw[thick, ->] (29/10, 3) -- (11/5, 3);
        \draw[thick, ->] (-21/10, 5) -- (-14/5, 5);
        \draw[thick, ->] (-2, 51/10) -- (-2, 34/5);
        \draw[thick, ->] (-21/10, 69/10) -- (-14/5, 26/5);
        \draw[thick, ->] (-19/10, 71/10) -- (-1/5, 39/5);
        \draw[thick, ->] (-1/10, 99/10) -- (-4/5, 1/5);
        \draw[thick, ->] (1/10, 99/10) -- (19/5, 46/5);
        \draw[thick, ->] (-1/10, 99/10) -- (-19/5, 46/5);
        \draw[thick, ->] (1/10, 99/10) -- (4/5, 1/5);
        \draw[thick, ->] (-39/10, 89/10) -- (-16/5, 16/5);
        \draw[thick, ->] (-39/10, 89/10) -- (-1/5, 41/5);
        \draw[thick, ->] (-39/10, 89/10) -- (-11/5, 36/5);
        \draw[thick, ->] (9/10, 0) -- (1/5, 0);
        \draw[thick, ->] (9/10, 1/10) -- (-9/5, 24/5);
        \draw[thick, ->] (11/10, 1/10) -- (9/5, 4/5);
        \draw[thick, ->] (2, 11/10) -- (2, 14/5);
        \draw[thick, ->] (21/10, 11/10) -- (14/5, 14/5);
    \end{tikzpicture}
    \caption{The digraph of $M_{4}$ for the orientation $\mathbf{c} = (912, 1824, 513, 817)$}
    \label{fig:Todd4}
\end{figure}

\begin{figure}
    \centering
    \begin{tikzpicture}
        \draw[thick] (0, 0) node[orange, circle,fill, inner sep = 1.5pt] {};
        \draw[thick] (-1, 0) node[orange, circle,fill, inner sep = 1.5pt] {};
        \draw[thick] (-2, 1) node[orange, circle,fill, inner sep = 1.5pt] {};
        \draw[thick] (1, 0) node[orange, circle,fill, inner sep = 1.5pt] {};
        \draw[thick] (2, 1) node[orange, circle,fill, inner sep = 1.5pt] {};
        \draw[thick] (2, 3) node[orange, circle,fill, inner sep = 1.5pt] {};
        \draw[thick] (3, 3) node[orange, circle,fill, inner sep = 1.5pt] {};
        \draw[thick] (-3, 3) node[orange, circle,fill, inner sep = 1.5pt] {};
        \draw[thick] (-2, 3) node[orange, circle,fill, inner sep = 1.5pt] {};
        \draw[thick] (-4, 9) node[orange, circle,fill, inner sep = 1.5pt] {};
        \draw[thick] (4, 9) node[orange, circle,fill, inner sep = 1.5pt] {};
        \draw[thick] (0, 3) node[orange, circle,fill, inner sep = 1.5pt] {};
        \draw[thick] (-2, 5) node[orange, circle,fill, inner sep = 1.5pt] {};
        \draw[thick] (-3, 5) node[orange, circle,fill, inner sep = 1.5pt] {};
        \draw[thick] (3, 5) node[orange, circle,fill, inner sep = 1.5pt] {};
        \draw[thick] (2, 5) node[orange, circle,fill, inner sep = 1.5pt] {};
        \draw[thick] (-2, 7) node[orange, circle,fill, inner sep = 1.5pt] {};
        \draw[thick] (2, 7) node[orange, circle,fill, inner sep = 1.5pt] {};
        \draw[thick] (0, 8) node[orange, circle,fill, inner sep = 1.5pt] {};
        \draw[thick] (0, 10) node[orange, circle,fill, inner sep = 1.5pt] {};
        \draw[thick, ->] (-1/10, 1/10) -- (-9/5, 4/5);
        \draw[thick, ->] (1/10, 1/10) -- (9/5, 4/5);
        \draw[thick, ->] (21/10, 5) -- (14/5, 5);
        \draw[thick, ->] (2, 51/10) -- (2, 34/5);
        \draw[thick, ->] (-3, 31/10) -- (-3, 24/5);
        \draw[thick, ->] (-29/10, 3) -- (-11/5, 3);
        \draw[thick, ->] (-9/10, 0) -- (-1/5, 0);
        \draw[thick, ->] (-9/10, 1/10) -- (9/5, 24/5);
        \draw[thick, ->] (-11/10, 1/10) -- (-9/5, 4/5);
        \draw[thick, ->] (-29/10, 51/10) -- (-11/5, 34/5);
        \draw[thick, ->] (29/10, 49/10) -- (1/5, 16/5);
        \draw[thick, ->] (-1/10, 31/10) -- (-14/5, 24/5);
        \draw[thick, ->] (-21/10, 11/10) -- (-14/5, 14/5);
        \draw[thick, ->] (-2, 11/10) -- (-2, 14/5);
        \draw[thick, ->] (21/10, 69/10) -- (14/5, 26/5);
        \draw[thick, ->] (19/10, 71/10) -- (1/5, 39/5);
        \draw[thick, ->] (-19/10, 31/10) -- (9/5, 24/5);
        \draw[thick, ->] (-19/10, 3) -- (-1/5, 3);
        \draw[thick, ->] (39/10, 89/10) -- (11/5, 36/5);
        \draw[thick, ->] (39/10, 89/10) -- (1/5, 41/5);
        \draw[thick, ->] (39/10, 89/10) -- (16/5, 16/5);
        \draw[thick, ->] (19/10, 3) -- (1/5, 3);
        \draw[thick, ->] (19/10, 31/10) -- (-9/5, 24/5);
        \draw[thick, ->] (3, 31/10) -- (3, 24/5);
        \draw[thick, ->] (29/10, 3) -- (11/5, 3);
        \draw[thick, ->] (-21/10, 5) -- (-14/5, 5);
        \draw[thick, ->] (-2, 51/10) -- (-2, 34/5);
        \draw[thick, ->] (-19/10, 71/10) -- (-1/5, 39/5);
        \draw[thick, ->] (-1/10, 99/10) -- (-4/5, 1/5);
        \draw[thick, ->] (1/10, 99/10) -- (19/5, 46/5);
        \draw[thick, ->] (-1/10, 99/10) -- (-19/5, 46/5);
        \draw[thick, ->] (1/10, 99/10) -- (4/5, 1/5);
        \draw[thick, ->] (-39/10, 89/10) -- (-16/5, 16/5);
        \draw[thick, ->] (-39/10, 89/10) -- (-1/5, 41/5);
        \draw[thick, ->] (-39/10, 89/10) -- (-11/5, 36/5);
        \draw[thick, ->] (9/10, 0) -- (1/5, 0);
        \draw[thick, ->] (9/10, 1/10) -- (-9/5, 24/5);
        \draw[thick, ->] (11/10, 1/10) -- (9/5, 4/5);
        \draw[thick, ->] (2, 11/10) -- (2, 14/5);
        \draw[thick, ->] (21/10, 11/10) -- (14/5, 14/5);
    \end{tikzpicture}
    \caption{The digraph of $M_{4}$ for the orientation $\mathbf{c} = (1824, 912, 817, 513)$}
    \label{fig:Todd5}
\end{figure}
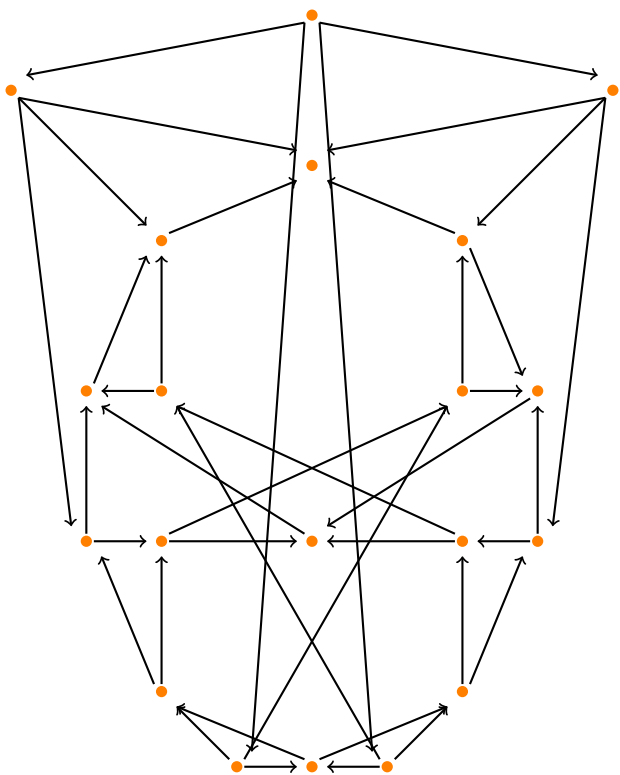

\begin{figure}[hbt]
\begin{subfigure}[b]{\textwidth}
\centering
\begin{tikzpicture}[every node/.style={inner sep=5pt},y=\yscale]
    \coordinate (v80) at (-1.2246467991473532e-16, -2.0);
    \coordinate (v42) at (-1.1755705045849463, -1.618033988749895);
    \coordinate (v44) at (-1.902113032590307, -0.6180339887498948);
    \coordinate (v120) at (-1.902113032590307, 0.6180339887498948);
    \coordinate (v108) at (-1.1755705045849463, 1.618033988749895);
    \coordinate (v171) at (-1.2246467991473532e-16, 2.0);
    \coordinate (v0) at (2, 0);
    \coordinate (v18) at (-6, -.6);
    \coordinate (v53) at (-4, 1);
    \coordinate (v26) at (-6, .6);
    \coordinate (v12) at (-8, 0);
    \coordinate (v95) at (-4, -1);

    \draw[thick] (v0) -- (v80);
    \draw[thick] (v0) -- (v171);
    \draw[thick] (v42) -- (v80);
    \draw[thick] (v42) -- (v44);
    \draw[thick] (v171) -- (v108);
    \draw[thick] (v12) -- (v18);
    \draw[thick] (v12) -- (v26);
    \draw[thick] (v44) -- (v120);
    \draw[thick] (v44) -- (v95);
    \draw[thick] (v108) -- (v120);
    \draw[thick] (v18) -- (v95);
    \draw[thick] (v18) -- (v26);
    \draw[thick] (v53) -- (v120);
    \draw[thick] (v53) -- (v95);
    \draw[thick] (v53) -- (v26);

    \draw[black,fill] (v0) circle (2pt) node[anchor=west] {$\mathbf{v}^-$};
    \draw[black,fill] (v80) circle (2pt) node[anchor=west,inner sep=7.5pt] {$\mathsf{2^-8^-}$};
    \draw[black,fill] (v42) circle (2pt) node[anchor=east] {$\mathsf{2^-7^+}$};
    \draw[blue,fill] (v44) circle (2pt) node[anchor=west] {$\mathsf{3^+7^+}$};
    \draw[blue,fill] (v120) circle (2pt) node[anchor=west] {$\mathsf{3^+6^+}$};
    \draw[black,fill] (v108) circle (2pt) node[anchor=east] {$\mathsf{2^+6^+}$};
    \draw[black,fill] (v171) circle (2pt) node[anchor=west,inner sep=7.5pt] {$\mathsf{2^+8^+}$};
    \draw[black,fill] (v12) circle (2pt) node[anchor=east] {$\mathbf{v}^+$};
    \draw[black,fill] (v26) circle (2pt) node[anchor=south east] {$\mathsf{1^+3^+4^+6^+10^+}$};
    \draw[black,fill] (v53) circle (2pt) node[anchor=south] {$\mathsf{3^+4^+6^+10^+13^+}$};
    \draw[black,fill] (v18) circle (2pt) node[anchor=north east] {$\mathsf{1^+3^+4^+7^+10^+}$};
    \draw[black,fill] (v95) circle (2pt) node[anchor=north] {$\mathsf{3^+4^+7^+10^+13^+}$};

    \node[gray] at (0,0) {$F \colon \mathsf{10^+11^+13^+}$};
\end{tikzpicture}
\caption{}
\label{fig:santos-23-face-a}
\end{subfigure}
\begin{subfigure}[b]{\textwidth}
\centering
\begin{tikzpicture}[every node/.style={inner sep=5pt},y=\yscale]
    \coordinate (v102) at (-1.2246467991473532e-16, -2.0);
    \coordinate (v164) at (-1.1755705045849463, -1.618033988749895);
    \coordinate (v175) at (-1.902113032590307, -0.6180339887498948);
    \coordinate (v103) at (-1.902113032590307, 0.6180339887498948);
    \coordinate (v167) at (-1.1755705045849463, 1.618033988749895);
    \coordinate (v170) at (-1.2246467991473532e-16, 2.0);
    \coordinate (v0) at (2, 0);
    \coordinate (v49) at (-6, -.6);
    \coordinate (v24) at (-4, -1);
    \coordinate (v41) at (-6, .6);
    \coordinate (v12) at (-8, 0);
    \coordinate (v45) at (-4, 1);

    \draw[thick] (v0) -- (v102);
    \draw[thick] (v0) -- (v170);
    \draw[thick] (v164) -- (v102);
    \draw[thick] (v164) -- (v175);
    \draw[thick] (v103) -- (v167);
    \draw[thick] (v103) -- (v175);
    \draw[thick] (v103) -- (v45);
    \draw[thick] (v167) -- (v170);
    \draw[thick] (v41) -- (v45);
    \draw[thick] (v41) -- (v12);
    \draw[thick] (v12) -- (v49);
    \draw[thick] (v45) -- (v24);
    \draw[thick] (v175) -- (v24);
    \draw[thick] (v49) -- (v24);

    \draw[black,fill] (v0) circle (2pt) node[anchor=west] {$\mathbf{v}^-$};
    \draw[black,fill] (v164) circle (2pt) node[anchor=east] {$\mathsf{ 3^- 6^- }$};
    \draw[black,fill] (v102) circle (2pt) node[anchor=west,inner sep=7.5pt] {$\mathsf{ 3^- 11^-}$};
    \draw[blue,fill] (v103) circle (2pt) node[anchor=west] {$\mathsf{ 2^+ 6^+ }$};
    \draw[black,fill] (v167) circle (2pt) node[anchor=east] {$\mathsf{ 3^+ 6^+ }$};
    \draw[black,fill] (v41) circle (2pt) node[anchor=south east] {$\mathsf{ 1^+ 2^+ 4^+ 6^+ 9^+}$};
    \draw[black,fill] (v170) circle (2pt) node[anchor=west,inner sep=7.5pt] {$\mathsf{ 3^+ 11^+}$};
    \draw[black,fill] (v12) circle (2pt) node[anchor=east] {$\mathbf{v}^+$};
    \draw[black,fill] (v45) circle (2pt) node[anchor=south] {$\mathsf{ 1^+ 2^+ 6^+ 9^+ 13^+}$};
    \draw[blue,fill] (v175) circle (2pt) node[anchor=west] {$\mathsf{ 2^+ 6^- }$};
    \draw[black,fill] (v49) circle (2pt) node[anchor=north east] {$\mathsf{ 1^+ 2^+ 4^- 6^- 9^+}$};
    \draw[black,fill] (v24) circle (2pt) node[anchor=north] {$\mathsf{ 1^+ 2^+ 6^- 9^+ 13^+}$};

    \node[gray] at (0,0) {$F \colon \mathsf{ 8^+ 9^+ 13^+}$};
\end{tikzpicture}
\caption{}
\label{fig:santos-23-face-b}
\end{subfigure}
\begin{subfigure}[b]{\textwidth}
\centering
\begin{tikzpicture}[every node/.style={inner sep=5pt},y=\yscale]  %,rotate=90
    \coordinate (v60) at (1.2246467991473532e-16, -2.0);
    \coordinate (v146) at (1.1755705045849463, -1.618033988749895);
    \coordinate (v93) at (1.902113032590307, -0.6180339887498948);
    \coordinate (v43) at (1.902113032590307, 0.6180339887498948);
    \coordinate (v45) at (1.1755705045849463, 1.618033988749895);
    \coordinate (v41) at (1.2246467991473532e-16, 2.0);
    \coordinate (v12) at (-2, 0);
    \coordinate (v0) at (8, 0);
    \coordinate (v137) at (4, -1);
    \coordinate (v108) at (4, 1);
    \coordinate (v171) at (6, .6);
    \coordinate (v158) at (6, -.6);

    \draw[thick] (v0) -- (v158);
    \draw[thick] (v0) -- (v171);
    \draw[thick] (v137) -- (v158);
    \draw[thick] (v137) -- (v93);
    \draw[thick] (v137) -- (v108);
    \draw[thick] (v41) -- (v45);
    \draw[thick] (v41) -- (v12);
    \draw[thick] (v43) -- (v108);
    \draw[thick] (v43) -- (v45);
    \draw[thick] (v43) -- (v93);
    \draw[thick] (v12) -- (v60);
    \draw[thick] (v108) -- (v171);
    \draw[thick] (v171) -- (v158);
    \draw[thick] (v146) -- (v60);
    \draw[thick] (v146) -- (v93);

    \draw[black,fill] (v0) circle (2pt) node[anchor=west] {$\mathbf{v}^-$};
    \draw[black,fill] (v137) circle (2pt) node[anchor=north] {$\mathsf{ 2^+ 6^+ 10^+ 11^+ 13^-}$};
    \draw[black,fill] (v41) circle (2pt) node[anchor=east] {$\mathsf{ 4^+ 9^+}$};
    \draw[blue,fill] (v43) circle (2pt) node[anchor=east] {$\mathsf{ 10^+ 13^+}$};
    \draw[black,fill] (v12) circle (2pt) node[anchor=east] {$\mathbf{v}^+$};
    \draw[black,fill] (v108) circle (2pt) node[anchor=south] {$\mathsf{ 2^+ 6^+ 10^+ 11^+ 13^+}$};
    \draw[black,fill] (v171) circle (2pt) node[anchor=south west] {$\mathsf{ 2^+ 8^+ 10^+ 11^+ 13^+}$};
    \draw[black,fill] (v45) circle (2pt) node[anchor=west,inner sep=7.5pt] {$\mathsf{ 9^+ 13^+}$};
    \draw[black,fill] (v146) circle (2pt) node[anchor=west,inner sep=7.5pt] {$\mathsf{ 9^- 13^-}$};
    \draw[black,fill] (v60) circle (2pt) node[anchor=east] {$\mathsf{5^+ 9^-}$};
    \draw[blue,fill] (v93) circle (2pt) node[anchor=east] {$\mathsf{ 10^+ 13^-}$};
    \draw[black,fill] (v158) circle (2pt) node[anchor=north west] {$\mathsf{ 2^+ 8^+ 10^+ 11^+ 13^-}$};

    \node[gray] at (0,0) {$F \colon \mathsf{ 1^+ 2^+ 6^+}$};
\end{tikzpicture}
\caption{}
\label{fig:santos-23-face-c}
\end{subfigure}
\begin{subfigure}[b]{\textwidth}
\centering
\begin{tikzpicture}[every node/.style={inner sep=5pt},y=\yscale]
    \coordinate (v88) at (1.2246467991473532e-16, -2.0);
    \coordinate (v215) at (1.1755705045849463, -1.618033988749895);
    \coordinate (v202) at (1.902113032590307, -0.6180339887498948);
    \coordinate (v116) at (1.902113032590307, 0.6180339887498948);
    \coordinate (v53) at (1.1755705045849463, 1.618033988749895);
    \coordinate (v26) at (1.2246467991473532e-16, 2.0);
    \coordinate (v12) at (-2, 0);
    \coordinate (v0) at (8, 0);
    \coordinate (v66) at (6, -.6);
    \coordinate (v197) at (4, -1);
    \coordinate (v167) at (4, 1);
    \coordinate (v170) at (6, .6);

    \draw[thick] (v0) -- (v66);
    \draw[thick] (v0) -- (v170);
    \draw[thick] (v66) -- (v197);
    \draw[thick] (v197) -- (v202);
    \draw[thick] (v197) -- (v167);
    \draw[thick] (v167) -- (v170);
    \draw[thick] (v167) -- (v116);
    \draw[thick] (v202) -- (v215);
    \draw[thick] (v202) -- (v116);
    \draw[thick] (v12) -- (v88);
    \draw[thick] (v12) -- (v26);
    \draw[thick] (v116) -- (v53);
    \draw[thick] (v53) -- (v26);
    \draw[thick] (v215) -- (v88);

    \draw[black,fill] (v88) circle (2pt) node[anchor=east] {$\mathsf{1^-10^-}$};
    \draw[black,fill] (v215) circle (2pt) node[anchor=west,inner sep=7.5pt] {$\mathsf{10^-14^+}$};
    \draw[blue,fill] (v202) circle (2pt) node[anchor=east] {$\mathsf{9^+14^+}$};
    \draw[blue,fill] (v116) circle (2pt) node[anchor=east] {$\mathsf{9^+13^+}$};
    \draw[black,fill] (v53) circle (2pt) node[anchor=west,inner sep=7.5pt] {$\mathsf{10^+13^+}$};
    \draw[black,fill] (v26) circle (2pt) node[anchor=east] {$\mathsf{1^+10^+}$};
    \draw[black,fill] (v12) circle (2pt) node[anchor=east] {$\mathbf{v}+$};
    \draw[black,fill] (v0) circle (2pt) node[anchor=west] {$\mathbf{v}^-$};
    \draw[black,fill] (v170) circle (2pt) node[anchor=south west] {$\mathsf{3^+8^+9^+11^+13^+}$};
    \draw[black,fill] (v167) circle (2pt) node[anchor=south] {$\mathsf{3^+6^+8^+9^+13^+}$};
    \draw[black,fill] (v66) circle (2pt) node[anchor=north west] {$\mathsf{3^+8^+9^+12^+14^+}$};
    \draw[black,fill] (v197) circle (2pt) node[anchor=north] {$\mathsf{3^+6^+8^+9^+14^+}$};

    \node[gray] at (0,0) {$F \colon \mathsf{3^+4^+6^+}$};
\end{tikzpicture}
\caption{}
\label{fig:santos-23-face-d}
\end{subfigure}

\caption{Four subgraphs of $S_5^{28}$ induced by all vertices of a $2$-face $F$ and by vertices on a shortest path from the apex not contained in $F$ to a vertex of $F$. The face $F$ is defined by inequalities (a) $\mathsf{10^+}$, $\mathsf{11^+}$, $\mathsf{13^+}$, or (b) $\mathsf{8^+}$, $\mathsf{9^+}$, $\mathsf{13^+}$, or (c) $\mathsf{1^+}$, $\mathsf{2^+}$, $\mathsf{6^+}$, or (d) $\mathsf{3^+}$, $\mathsf{4^+}$, $\mathsf{6^+}$, respectively. Vertex labels indicate which inequalities are tight. For the vertices of $F$, we omitted the facets containing $F$ from their labels. In each of the graphs (a) to (d), the two highlighted vertices are at distance $3$ from both $\mathbf{v}^+$ and $\mathbf{v}^-$.}
\label{fig:santos-23-face}
\end{figure}

\begin{figure}[hbt]
\begin{subfigure}[b]{\textwidth}
\centering
\begin{tikzpicture}[every node/.style={inner sep=5pt},y=\yscale]
    \coordinate (v127) at (-1.2246467991473532e-16, 2.0);
    \coordinate (v132) at (-1.1755705045849463, 1.618033988749895);
    \coordinate (v130) at (-1.902113032590307, 0.6180339887498948);
    \coordinate (v138) at (-1.902113032590307, -0.6180339887498948);
    \coordinate (v131) at (-1.1755705045849463, -1.618033988749895);
    \coordinate (v151) at (-1.2246467991473532e-16, -2.0);
    \coordinate (v0) at (2, 0);
    \coordinate (v50) at (-8, 0);
    \coordinate (v62) at (-6, .6);
    \coordinate (v9) at (-4, 1);
    \coordinate (v46) at (-6, -.6);
    \coordinate (v41) at (-4, -1);

    \draw[thick] (v0) -- (v151);
    \draw[thick] (v0) -- (v127);
    \draw[thick] (v130) -- (v9);
    \draw[thick] (v130) -- (v138);
    \draw[thick] (v130) -- (v132);
    \draw[thick] (v131) -- (v138);
    \draw[thick] (v131) -- (v151);
    \draw[thick] (v132) -- (v127);
    \draw[thick] (v9) -- (v62);
    \draw[thick] (v9) -- (v41);
    \draw[thick] (v138) -- (v41);
    \draw[thick] (v41) -- (v46);
    \draw[thick] (v46) -- (v50);
    \draw[thick] (v50) -- (v62);

    \draw[black,fill] (v0) circle (2pt) node[anchor=west] {$\mathbf{v}^-$};
    \draw[blue,fill] (v130) circle (2pt) node[anchor=west] {$\mathsf{ 3\;5}$};
    \draw[black,fill] (v131) circle (2pt) node[anchor=east,inner sep=7.5pt] {$\mathsf{ 2\;7}$};
    \draw[black,fill] (v132) circle (2pt) node[anchor=east,inner sep=7.5pt] {$\mathsf{ 1\;5}$};
    \draw[black,fill] (v9) circle (2pt) node[anchor=south] {$\mathsf{3\;5\;9\;15\;17}$};
    \draw[blue,fill] (v138) circle (2pt) node[anchor=west] {$\mathsf{ 3\;7}$};
    \draw[black,fill] (v41) circle (2pt) node[anchor=north] {$\mathsf{ 3\;7\;9\;15\;17}$};
    \draw[black,fill] (v46) circle (2pt) node[anchor=north east] {$\mathsf{ 3\;7\;9\;11\;15}$};
    \draw[black,fill] (v50) circle (2pt) node[anchor=east] {$\mathbf{v}^+$};
    \draw[black,fill] (v151) circle (2pt) node[anchor=west,inner sep=7.5pt] {$\mathsf{ 2\;22}$};
    \draw[black,fill] (v62) circle (2pt) node[anchor=south east] {$\mathsf{3\;5\;9\;12\;15}$};
    \draw[black,fill] (v127) circle (2pt) node[anchor=west,inner sep=7.5pt] {$\mathsf{ 1\;23}$};

    \node[gray] at (0,0) {$F \colon \mathsf{15\; 17\; 21}$};
\end{tikzpicture}
\caption{}
\label{fig:weibel-face-a}
\end{subfigure}
\begin{subfigure}[b]{\textwidth}
\centering
\begin{tikzpicture}[every node/.style={inner sep=5pt},y=\yscale]
    \coordinate (v153) at (-1.2246467991473532e-16, 2.0);
    \coordinate (v172) at (-0.9999999999999998, 1.7320508075688774);
    \coordinate (v160) at (-1.7320508075688772, 1.0);
    \coordinate (v161) at (-2.0, -0.0);
    \coordinate (v174) at (-1.7320508075688774, -0.9999999999999997);
    \coordinate (v175) at (-0.9999999999999998, -1.7320508075688774);
    \coordinate (v205) at (-1.2246467991473532e-16, -2.0);
    \coordinate (v0) at (2, 0);
    \coordinate (v165) at (-4, -1);
    \coordinate (v50) at (-8, 0);
    \coordinate (v21) at (-4, 1);
    \coordinate (v29) at (-6, .6);
    \coordinate (v30) at (-6, -.6);

    \draw[thick] (v160) -- (v21);
    \draw[thick] (v160) -- (v172);
    \draw[thick] (v160) -- (v161);
    \draw[thick] (v161) -- (v21);
    \draw[thick] (v161) -- (v165);
    \draw[thick] (v161) -- (v174);
    \draw[thick] (v0) -- (v153);
    \draw[thick] (v0) -- (v205);
    \draw[thick] (v165) -- (v30);
    \draw[thick] (v165) -- (v174);
    \draw[thick] (v172) -- (v153);
    \draw[thick] (v205) -- (v175);
    \draw[thick] (v174) -- (v175);
    \draw[thick] (v50) -- (v29);
    \draw[thick] (v50) -- (v30);
    \draw[thick] (v21) -- (v29);
    \draw[thick] (v29) -- (v30);

    \draw[blue,fill] (v160) circle (2pt) node[anchor=west,inner sep=7.5pt] {$\mathsf{ 2\;7}$};
    \draw[blue,fill] (v161) circle (2pt) node[anchor=west] {$\mathsf{ 2\;9}$};
    \draw[black,fill] (v0) circle (2pt) node[anchor=west] {$\mathbf{v}^-$};
    \draw[black,fill] (v165) circle (2pt) node[anchor=north] {$\mathsf{ 2\;8\;9\;13\;22}$};
    \draw[black,fill] (v172) circle (2pt) node[anchor=east,inner sep=7.5pt] {$\mathsf{ 3\;7}$};
    \draw[black,fill] (v205) circle (2pt) node[anchor=west,inner sep=7.5pt] {$\mathsf{ 4\;23}$};
    \draw[blue,fill] (v174) circle (2pt) node[anchor=west,inner sep=7.5pt] {$\mathsf{ 8\;9}$};
    \draw[black,fill] (v175) circle (2pt) node[anchor=east,inner sep=7.5pt] {$\mathsf{ 4\;8}$};
    \draw[black,fill] (v50) circle (2pt) node[anchor=east] {$\mathbf{v}^+$};
    \draw[black,fill] (v21) circle (2pt) node[anchor=south] {$\mathsf{ 2\;7\;9\;13\;18}$};
    \draw[black,fill] (v153) circle (2pt) node[anchor=west,inner sep=7.5pt] {$\mathsf{ 3\;21}$};
    \draw[black,fill] (v29) circle (2pt) node[anchor=south east] {$\mathsf{ 2\;7\;9\;11\;13}$};
    \draw[black,fill] (v30) circle (2pt) node[anchor=north east] {$\mathsf{ 2\;8\;9\;11\;13}$};

    \node[gray] at (0,0) {$F \colon \mathsf{13\; 18\; 22}$};
\end{tikzpicture}
\caption{}
\label{fig:weibel-face-b}
\end{subfigure}
\begin{subfigure}[b]{\textwidth}
\centering
\begin{tikzpicture}[every node/.style={inner sep=5pt},y=\yscale]
    \coordinate (v115) at (1.2246467991473532e-16, -2.0);
    \coordinate (v107) at (0.9999999999999998, -1.7320508075688774);
    \coordinate (v86) at (1.7320508075688772, -1.0);
    \coordinate (v69) at (2.0, 0.0);
    \coordinate (v176) at (1.7320508075688774, 0.9999999999999997);
    \coordinate (v165) at (0.9999999999999998, 1.7320508075688774);
    \coordinate (v30) at (1.2246467991473532e-16, 2.0);
    \coordinate (v50) at (-2, 0);
    \coordinate (v0) at (8, 0);
    \coordinate (v177) at (4, 1);
    \coordinate (v146) at (6, .6);
    \coordinate (v145) at (6, -.6);
    \coordinate (v184) at (4, -1);

    \draw[thick] (v0) -- (v146);
    \draw[thick] (v0) -- (v145);
    \draw[thick] (v69) -- (v176);
    \draw[thick] (v69) -- (v86);
    \draw[thick] (v69) -- (v177);
    \draw[thick] (v165) -- (v30);
    \draw[thick] (v165) -- (v176);
    \draw[thick] (v107) -- (v86);
    \draw[thick] (v107) -- (v115);
    \draw[thick] (v176) -- (v177);
    \draw[thick] (v177) -- (v184);
    \draw[thick] (v177) -- (v146);
    \draw[thick] (v146) -- (v145);
    \draw[thick] (v145) -- (v184);
    \draw[thick] (v50) -- (v115);
    \draw[thick] (v50) -- (v30);
    \draw[thick] (v86) -- (v184);

    \draw[black,fill] (v0) circle (2pt) node[anchor=west] {$\mathbf{v}^-$};
    \draw[blue,fill] (v69) circle (2pt) node[anchor=east] {$\mathsf{ 16\;18}$};
    \draw[black,fill] (v165) circle (2pt) node[anchor=west,inner sep=7.5pt] {$\mathsf{ 13\;22}$};
    \draw[black,fill] (v107) circle (2pt) node[anchor=west,inner sep=7.5pt] {$\mathsf{ 14\;20}$};
    \draw[blue,fill] (v176) circle (2pt) node[anchor=east,inner sep=7.5pt] {$\mathsf{ 18\;22}$};
    \draw[black,fill] (v177) circle (2pt) node[anchor=south] {$\mathsf{ 2\;8\;16\;18\;22}$};
    \draw[black,fill] (v146) circle (2pt) node[anchor=south west] {$\mathsf{ 2\;16\;18\;21\;22}$};
    \draw[black,fill] (v145) circle (2pt) node[anchor=north west] {$\mathsf{ 2\;16\;20\;21\;22}$};
    \draw[black,fill] (v50) circle (2pt) node[anchor=east] {$\mathbf{v}^+$};
    \draw[black,fill] (v115) circle (2pt) node[anchor=east,inner sep=7.5pt] {$\mathsf{ 10\;14}$};
    \draw[blue,fill] (v86) circle (2pt) node[anchor=east,inner sep=7.5pt] {$\mathsf{ 16\;20}$};
    \draw[black,fill] (v184) circle (2pt) node[anchor=north] {$\mathsf{ 2\;8\;16\;20\;22}$};
    \draw[black,fill] (v30) circle (2pt) node[anchor=east,inner sep=7.5pt] {$\mathsf{ 11\;13}$};

    \node[gray] at (0,0) {$F \colon \mathsf{ 2\; 8\; 9}$};
\end{tikzpicture}
\caption{}
\label{fig:weibel-face-c}
\end{subfigure}
\begin{subfigure}[b]{\textwidth}
\centering
\begin{tikzpicture}[every node/.style={inner sep=5pt},y=\yscale]
    \coordinate (v77) at (1.2246467991473532e-16, -2.0);
    \coordinate (v65) at (1.1755705045849463, -1.618033988749895);
    \coordinate (v8) at (1.902113032590307, -0.6180339887498948);
    \coordinate (v25) at (1.902113032590307, 0.6180339887498948);
    \coordinate (v9) at (1.1755705045849463, 1.618033988749895);
    \coordinate (v62) at (1.2246467991473532e-16, 2.0);
    \coordinate (v50) at (-2, 0);
    \coordinate (v0) at (8, 0);
    \coordinate (v2) at (6, -.6);
    \coordinate (v211) at (4, -1);
    \coordinate (v215) at (4, 1);
    \coordinate (v121) at (6, .6);

    \draw[thick] (v0) -- (v2);
    \draw[thick] (v0) -- (v121);
    \draw[thick] (v65) -- (v8);
    \draw[thick] (v65) -- (v77);
    \draw[thick] (v2) -- (v211);
    \draw[thick] (v2) -- (v121);
    \draw[thick] (v121) -- (v215);
    \draw[thick] (v8) -- (v25);
    \draw[thick] (v8) -- (v211);
    \draw[thick] (v9) -- (v62);
    \draw[thick] (v9) -- (v25);
    \draw[thick] (v77) -- (v50);
    \draw[thick] (v50) -- (v62);
    \draw[thick] (v211) -- (v215);
    \draw[thick] (v215) -- (v25);

    \draw[black,fill] (v0) circle (2pt) node[anchor=west] {$\mathbf{v}^-$};
    \draw[black,fill] (v65) circle (2pt) node[anchor=west,inner sep=7.5pt] {$\mathsf{ 16\;18}$};
    \draw[black,fill] (v2) circle (2pt) node[anchor=north west] {$\mathsf{ 3\;13\;18\;21\;23}$};
    \draw[black,fill] (v121) circle (2pt) node[anchor=south west] {$\mathsf{ 3\;13\;17\;21\;23}$};
    \draw[blue,fill] (v8) circle (2pt) node[anchor=east,inner sep=7.5pt] {$\mathsf{ 13\;18}$};
    \draw[black,fill] (v9) circle (2pt) node[anchor=west,inner sep=7.5pt] {$\mathsf{ 15\;17}$};
    \draw[black,fill] (v77) circle (2pt) node[anchor=east,inner sep=7.5pt] {$\mathsf{ 10\;16}$};
    \draw[black,fill] (v50) circle (2pt) node[anchor=east] {$\mathbf{v}^+$};
    \draw[black,fill] (v211) circle (2pt) node[anchor=north] {$\mathsf{ 3\;5\;13\;18\;23}$};
    \draw[black,fill] (v215) circle (2pt) node[anchor=south] {$\mathsf{ 3\;5\;13\;17\;23}$};
    \draw[blue,fill] (v25) circle (2pt) node[anchor=east,inner sep=7.5pt] {$\mathsf{ 13\;17}$};
    \draw[black,fill] (v62) circle (2pt) node[anchor=east,inner sep=7.5pt] {$\mathsf{ 12\;15}$};

    \node[gray] at (0,0) {$F \colon \mathsf{ 3 \;5\; 9}$};
\end{tikzpicture}
\caption{}
\label{fig:weibel-face-d}
\end{subfigure}
\caption{Four subgraphs of $S_5^{25}$ induced by all vertices of a $2$-face $F$ and by vertices on a shortest path from the apex not contained in $F$ to a vertex of $F$. The face $F$ is defined by inequalities (a) $\mathsf{15}, \mathsf{17}, \mathsf{21}$, or (b) $\mathsf{13}, \mathsf{18}, \mathsf{22}$, or (c) $\mathsf{2}, \mathsf{8}, \mathsf{9}$, or (d) $\mathsf{3}, \mathsf{5}, \mathsf{9}$, respectively. Vertex labels indicate which inequalities are tight. For the vertices of $F$, we omitted the facets containing $F$ from their labels. In each of the graphs (a) to (d), the highlighted vertices are at distance $3$ from the apex not contained in $F$.}
\label{fig:weibel-face}
\end{figure}
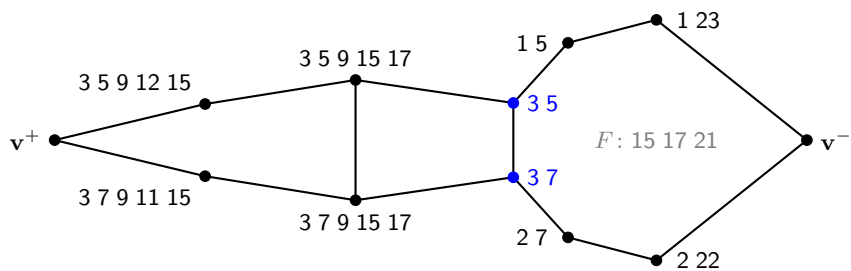
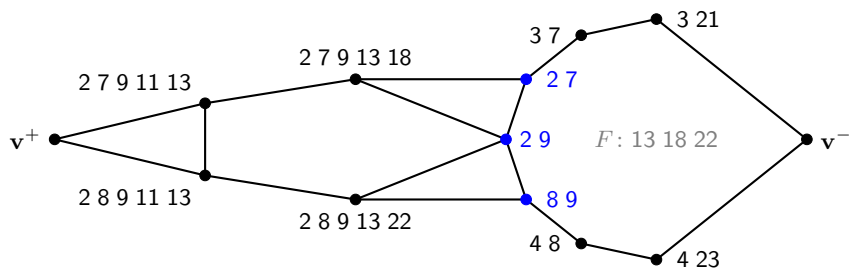
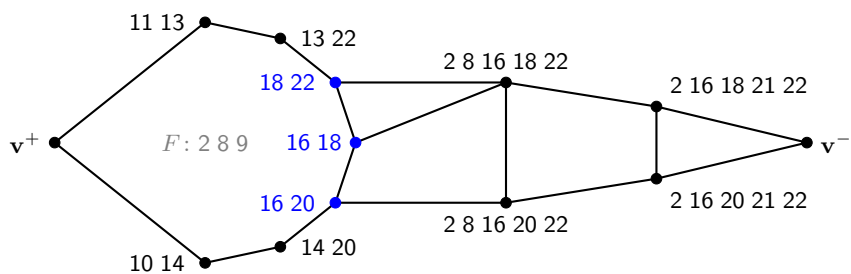
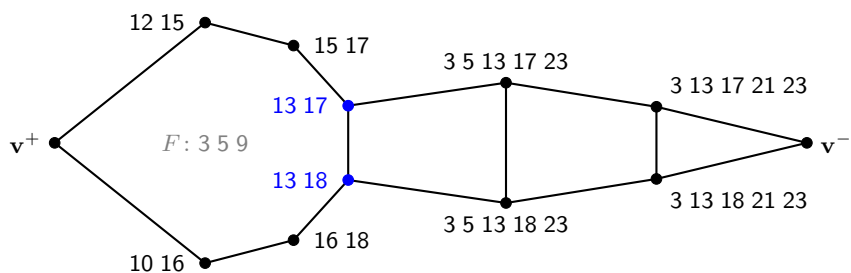

%---------------------------------------------------------------------

\end{document}